\documentclass{article}

\usepackage[T1]{fontenc}
\usepackage{mathptmx}

\usepackage{setspace}
\usepackage{amsfonts}
\usepackage{amsmath}
\usepackage{amssymb}
\usepackage{amsthm}
\usepackage{mathrsfs}
\usepackage{mathtools}
\usepackage{slashed}
\usepackage{enumerate}
\usepackage{hyperref}
\usepackage{cleveref}
\usepackage{tikz,tikz-cd}
\usepackage{extpfeil}
\usepackage{hhline}
\usepackage{float}
\usepackage{thmtools}
\usepackage{thm-restate}
\usepackage{xcolor}

\usepackage{authblk}

\hypersetup{
    colorlinks=false,
    linkcolor=violet,
    filecolor=magenta,      
    urlcolor=cyan,
    citecolor=red,
    }

\makeatletter
\newcommand*{\doublerightarrow}[2]{\mathrel{
  \settowidth{\@tempdima}{$\scriptstyle#1$}
  \settowidth{\@tempdimb}{$\scriptstyle#2$}
  \ifdim\@tempdimb>\@tempdima \@tempdima=\@tempdimb\fi
  \mathop{\vcenter{
    \offinterlineskip\ialign{\hbox to\dimexpr\@tempdima+1em{##}\cr
    \rightarrowfill\cr\noalign{\kern.5ex}
    \rightarrowfill\cr}}}\limits^{\!#1}_{\!#2}}}

\newcommand{\img}[2][1]{\begin{gathered}\includegraphics[scale=#1]{#2}\end{gathered}}
\newcommand{\revLpic}{\raisebox{-1pt}{$\img{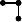}$}}
\newcommand{\Lpic}{\raisebox{-1pt}{$\img{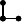}$}}
\newcommand{\squarepic}{\raisebox{-1pt}{$\img{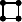}$}}

\newcommand\KK{\mathbf{K}}
\newcommand\LL{\mathbf{L}}

\newcommand\CC{\mathbb{C}}

\newcommand\PP{\mathbb{P}}

\newcommand\NN{\mathbb{N}}
\newcommand\ZZ{\mathbb{Z}}
\newcommand\GG{\mathbb{G}}
\numberwithin{equation}{section}

\theoremstyle{plain}
\newtheorem{theorem}{Theorem}[section]

\newtheorem{proposition}[theorem]{Proposition}
\newtheorem{corollary}[theorem]{Corollary}
\newtheorem{lemma}[theorem]{Lemma}

\theoremstyle{definition}
\newtheorem{definition}[theorem]{Definition}
\newtheorem{example}[theorem]{Example}

\theoremstyle{definition}
\newtheorem{remark}[theorem]{Remark}

\newtheorem{question}[theorem]{Question}
\newtheorem{problem}[theorem]{Problem}

\newcommand{\del}{\partial}
\newcommand{\delbar}{{\overline{\partial}}}
\newcommand{\mm}{\mathcal}
\newcommand{\GL}{\mathrm{GL}}
\newcommand{\cone}{\mathrm{Cone}}

\DeclareMathOperator{\coker}{{\mathrm{coker}}}
\DeclareMathOperator{\im}{{\mathrm{im}}}
\DeclareMathOperator{\aut}{{\mathrm{Aut}}}
\title{Obstruction Theory for Bigraded Differential Algebras}

\author{Jiahao Hu}
\date{}
\begin{document}
\maketitle
\begin{abstract}
We develop an obstruction theory for Hirsch extensions of cbba's with twisted coefficients. This leads to a variety of applications, including a structural theorem for minimal cbba's, a construction of relative minimal models, as well as a proof of uniqueness. These results are further employed to study automorphism groups of minimal cbba's and to characterize formality in terms of grading automorphisms.
\end{abstract}
\tableofcontents

\section*{Introduction}
\addcontentsline{toc}{section}{Introduction}
The purpose of this paper is to develop an obstruction theory for Hirsch extensions of cbba's (commutative bigraded bidifferential algebras) in order to study their homotopy theory. The corresponding theory for cdga's (commutative differential graded algebras) has led to a flood of breakthroughs in geometry and topology through a non-exhaustible list of work, starting with \cite{dgsm} and \cite{infinitesimal}. The homotopy theory, known as pluripotetial homotopy theory, for cbba under consideration here, put forth recently by Stelzig in \cite{stelzig2025pluripotential}, aims to extend such success towards understanding complex manifolds via their bigraded de Rham algebras of complex differential forms. In \cite{stelzig2025pluripotential} minimal models for cbba's are constructed in the Aeppli-simply-connected case which then leads to new holomorphic invariants of complex manifolds such as refinements of the de Rham homotopy groups.

However, the construction of minimal models in \cite{stelzig2025pluripotential} is based on a complicated analysis of various cohomology groups for bicomplexes, and has to go through a somewhat mysterious "re-minimization" step. Here we decide to offer an account of obstruction theory in order to complete and simplify the pluripotential homotopy theory of cbba. It turns out that once the obstruction theory for Hirsch extensions is in place, a neat construction of minimal models, among many other applications, soon follows.

We here emphasize the similarity between the homotopy theory of cbba and that of cdga--both theories rely heavily, if not entirely, on the following abstract categorical structures:
\begin{enumerate}[(i)]
	\item a triangulated category $\mathcal{C}$ (e.g. of chain complexes, bicomplexes) with a (homotopy) t-structure $\mathcal{C}_\ge 0,\mathcal{C}_\le 0$,
	\item symmetric monoidal operations $\oplus,\otimes$ compatible with the structures in (i), in particular $\oplus$ and $\otimes$ preserve $\mathcal{C}_{\ge 0}$;
	\item an exponential free functor, turning $\oplus$ into $\otimes$;
	\item a cohomology functor derived from t-structure $H:\mathcal{C}\to\mathcal{C}_\heartsuit$ which splits on object level, giving a theory of additive minimal models.
\end{enumerate}
With these, the additive homotopy theory on $\mathcal{C}$ can be promoted into a multiplicative homotopy theory on algebra objects in $\mathcal{C}$--additive minimal model upgraded to multiplicative minimal model, cohomology ungraded to cohomotopy and so on. The reader is recommended to keep an eye on where and how these structures are used.

\subsection*{Acknowledgments}
I thank Aleksandar Milivojevi\'c and Jonas Stelzig for useful discussions.

\section{Homotopy theory of bicomplex}
Throughout we work over a field $\KK$ of characteristic zero. The reader is assumed to be familiar with the homotopy theory of bicomplexes, see \cite{stelzig2021structure} for a reference.
\subsection{Preliminaries}
A bicomplex is a bigraded vector space equipped with two anti-commuting differentials $\del,\delbar$ of degrees $(1,0)$ and $(0,1)$ respectively. To a bicomplex, we associate various cohomology groups:
\begin{itemize}
	\item (Bott-Chern) $H_{BC}=\ker\del\cap\ker\delbar/\im\del\delbar$,
	\item (reduced Bott-Chern) $H_{\widetilde{BC}}=\ker \del \cap \ker \delbar\cap (\im \del+\im \delbar)/\im \del\delbar$,
	\item $H_{dot}=\ker \del\cap\ker \delbar/\ker \del \cap \ker \delbar\cap (\im \del+\im \delbar)$,
	\item (Aeppli) $H_A=\ker\del\delbar/(\im\del+\im\delbar)$,
	\item (reduced Aeppli) $H_{\widetilde{A}}=\ker \del \delbar/(\ker \del\cap\ker \delbar+\im \del+\im \delbar)$,
	\item (Dolbeault) $H_\del=\ker\del/\im\del$, $H_\delbar=\ker\delbar/\im\delbar$.
\end{itemize}
The Bott-Chern cohomology is differed from reduced Bott-Chern by $H_{dot}$; similarly Aeppli is differed from reduced Aeppli by $H_{dot}$. These groups come equipped with induced bigradings and differentials, such as:
\[
H_{\widetilde{A}}\doublerightarrow{\del}{\delbar} H_{BC}.
\]

A bicomplex morphism is said to be a quasi-isomorphism if it induces an isomorphism on both Bott-Chern and Aeppli cohomology. A bicomplex is said to be contractible if its Bott-Chern and Aeppli cohomology vanish. Every contractible bicomplex is a direct sum of squares:
\[
\begin{tikzcd}[row sep=small, column sep=small]
	\delbar x\ar[r,"\del"] & \del\delbar x\\
	x\ar[u,"\delbar"]\ar[r,"\del"] & \del x\ar[u,"\delbar"]
\end{tikzcd}
\]
More generally every bicomplex is a direct sum of squares and zig-zags (\cite{stelzig2021structure}\cite{khovanov2020faithful}), see \Cref{tensor-product} for depictions of zig-zags. The zig-zag part of the direct sum decomposition can be recovered from cohomology--the zig-zag part is (non-canonically) isomorphic to $H_{\widetilde{A}}+H_{BC}$ and $H_{A}+H_{\widetilde{BC}}$, with induced differentials.

A bicomplex is said to be minimal if $\del\delbar=0$. A minimal model of a bicomplex is a minimal bicomplex quasi-isomorphic to it. Any two minimal models are isomorphic. In fact, the minimal model of a bicomplex is isomorphic to the zig-zag part of the aformentioned decomposition.

A square will be denoted by $\squarepic$ or $\square$, and we denote the following bicomplexes by $\revLpic$, $\bullet$ and $\Lpic$ respectively. These pictures are borrowed from \cite{stelzig2025pluripotential}.
\begin{center}
\begin{tabular}{ c c c c c }
$\begin{tikzcd}[row sep=small, column sep=small]
	y & x \\
	& z
	\arrow[from=1-1, to=1-2]
	\arrow[from=2-2, to=1-2]
\end{tikzcd}$ & &
 $x$ & &
 $\begin{tikzcd}[row sep=small, column sep=small]
	\delbar x \\
	x & \del x
	\arrow[from=2-1, to=1-1]
	\arrow[from=2-1, to=2-2]
\end{tikzcd}$
\end{tabular}
\end{center}
The convention here is that, each dot represents a one-dimensional vector space, horizontal and vertical arrows represent $\del$ and $\delbar$ respectively, each of which is an isomorphism; undrawn arrows vanish. 

\subsection{Truncation}
We use upper indices for degrees, for example $V^k$, $V^{p,q}$.
\begin{definition}[truncation]
	For a bicomplex $V$, we define its \textbf{truncations} $\tau_{\le k}V$ and $\tau_{\ge k}V$ to be
\begin{align*}
	(\tau_{\le k}V)^i&=
\begin{cases}
	V^i & i\le k-1\\
	(\ker \del\delbar)^k & i=k\\
	(\ker \del\cap\ker\delbar\cap(\im\del+\im\delbar))^{k+1} & i=k+1\\
	0 & i\ge k+2
\end{cases}\\
(\tau_{\ge k}V)^i &=
\begin{cases}
	0 & i\le k-1\\
	V^k/(\im\del+\im\delbar)^k & i=k\\
	V^{k+1}/(\im \del\delbar)^{k+1} & i=k+1\\
	V^i & i\ge k+2
\end{cases}
\end{align*}
\end{definition}
\begin{definition}[cohomology]
Define the \textbf{$k$-th cohomology bicomplex} $H^k$ to be \[H^k:=\tau_{\le k}\tau_{\ge k}=\tau_{\ge k}\tau_{\le k}=H^k_{A}\oplus H^{k+1}_{\widetilde{BC}}.\]
Therefore $HV=\oplus_k H^k V$ is a minimal model of $V	$.
\end{definition}

\begin{definition}[connectedness]\label{defn:connectedness}
	A bicomplex $V$ is said to be \textbf{$k$-connected} if one of the following equivalent conditions hold:
	\begin{enumerate}[(i)]
		\item $\tau_{\le k}V$ is contractible,
		\item $V\to \tau_{\ge k+1}V$ is a quasi-isomorphism,
		\item $H^{\le k}V=0$,
		\item $H_A^{\le k}V=0$.
	\end{enumerate}
\end{definition}

\begin{lemma}
	\begin{enumerate}[(i)]
		\item $H^{\ge k+1}(\tau_{\le k}V)=0$ and $\tau_{\le k} V\to V$ induces an isomorphism on $H^{\le k}$.
		\item $H^{\le k-1}(\tau_{\ge k}V)=0$ and $V\to \tau_{\ge k}V$ induces an isomorphism on $H^{\ge k}$.
	\end{enumerate}
\end{lemma}
\begin{proof}
	Straightforward verifications.
\end{proof}

\begin{lemma}[connectivity of tensor product]\label{connectivity}
	Let $V$ be $(k-1)$-connected and $W$ be $(l-1)$-connected, then $V\otimes W$ is $(k+l-1)$-connected. 
\end{lemma}
\begin{proof}
	Replacing $V$ by $\tau_{\ge k}V$ if necessary, we may assume $V$ is concentrated in degrees $\ge k$. Similarly, we can assume $W$ is concentrated in degrees $\ge l$ and therefore $V\otimes W$ is concentrated in degrees $\ge k+l$. This in turn implies $V\otimes W$ is $(k+l-1)$-connected.
\end{proof}
\begin{remark}
	Tensor products of zig-zags can be explicitly described, see \Cref{tensor-product}.
\end{remark}
\begin{corollary}[wedge product]\label{wedge-product} Let $V$ be a bicomplex. If $V$ is $(k-1)$-connected, then $\bigwedge ^n V$ is $(nk-1)$-connected.
\end{corollary}
\begin{proof}
This follows from that $\bigwedge^n V$ is a direct summand of $V^{\otimes n}$ as bicomplex so that its cohomology is a direct summand of $H(V^{\otimes n})$. Since $V^{\otimes n}$ is $(nk-1)$-connected by the above proposition, so is $\bigwedge^n V$.
\end{proof}

\subsection{Distinguished triangle}
The \textbf{shift operators} on bicomplexes are given by
\begin{align*}
	V[1]&:= \Lpic\otimes V\\
V[-1]&:= \revLpic\otimes V
\end{align*}
which are homotopy inverses; indeed the natural inclusion $\bullet\to \Lpic\otimes\revLpic$, mapping onto the tensor product of the corners, is a quasi-isomorphism. Here the conners of $\Lpic$ and $\revLpic$ are in degrees $(-1,-1)$, $(1,1)$ respectively.
%
For a map $f: V\to W$, its \textbf{mapping cone} is defined by
\[
\cone(f):=\coker(V\to W\oplus\squarepic\otimes V)
\]
in which the upper right conner of $\squarepic$ is in degree (0,0) and the map $V\to W\oplus \squarepic \otimes V$ is given by $f$ and the inclusion $\bullet\to \squarepic$ to the upper right corner. Under the linear isomorphism $\cone(f)=W\oplus V[1]$, the differentials are
\[
	\del\left(w, \begin{tikzcd}[row sep=tiny, column sep=tiny]
	c \\
	a & b
	\arrow[no head, from=2-1, to=1-1]
	\arrow[no head, from=2-1, to=2-2]
\end{tikzcd}\right)
=
\left(\del w-fc,  \begin{tikzcd}[row sep=tiny, column sep=tiny]
	-\del c \\
	\del a & a-\del b
	\arrow[no head, from=2-1, to=1-1]
	\arrow[no head, from=2-1, to=2-2]
\end{tikzcd}\right)
\]
\[
\delbar\left(w, \begin{tikzcd}[row sep=tiny, column sep=tiny]
	c \\
	a & b
	\arrow[no head, from=2-1, to=1-1]
	\arrow[no head, from=2-1, to=2-2]
\end{tikzcd}\right)
=
\left(\delbar w+fb,  \begin{tikzcd}[row sep=tiny, column sep=tiny]
	a-\delbar c \\
	\delbar a & -\delbar b
	\arrow[no head, from=2-1, to=1-1]
	\arrow[no head, from=2-1, to=2-2]
\end{tikzcd}\right)
\]
A \textbf{distinguished triangle} is one quasi-isomorphic to
\[
V\xrightarrow{f} W\to \cone(f).
\]
A distinguished triangle can be extended in both directions by applying shift operators
\[
\dots\to W[-1]\to \cone(f)[-1]\to V\xrightarrow{f} W\to \cone(f)\to V[1]\to W[1]\to\cdots
\]
and consequently induce long exact sequences in homotopy, see \cite[Corollary 11.2]{stelzig2025pluripotential}.

\begin{definition}
	A map $f: V\to W$ is said to be \textbf{$k$-connected} if $\cone(f)$ is $(k-1)$-connected.
\end{definition}
\begin{lemma}
	$f: V\to W$ is $k$-connected iff $H_A^{\le k-1}(f)$ is surjective and $H_{BC}^{\le k+1}(f)$ is injective.
\end{lemma}
\begin{proof}
	This follows from the long exact sequence \cite[Corollary 11.2]{stelzig2025pluripotential} derived from the distinguished triangle $V\to W\to\cone(f)$:
	\[
	\cdots\to H^i_A V\to H^i_A W\to H^i_A\cone(f)\to H^{i+2}_{BC} V \to H^{i+2}_{BC} W\to\cdots
	\]
\end{proof}
\begin{remark}\label{Aeppli of shift}
	Above we used the isomorphism $H_{A}^{p,q} (V[1])\equiv H_{BC}^{p+1,q+1} (V)$ and thus $H_A^i (V[1])\equiv H^{i+2}_{BC} (V)$. 
\end{remark}
\begin{remark}
	Connectivity is controlled by $G^k:=H^k_{\widetilde{A}}+H^{k+1}_{BC}$. Indeed a map $f$ is $k$-connected iff $G^i(f)$ is an isomorphism for $i<k$ and surjective for $i=k$.
\end{remark}

We shall encounter the following examples of distinguished triangles.
\begin{example}[distinguished triangle]\label{special-triangle}
	\begin{enumerate}[(i)]
		\item Let $0\to V\xrightarrow{f} W\to \coker(f)\to 0$ be a short exact sequence of bicomplexes. Then $V\to W \to \coker(f)$ is a distinguished triangle. To see this, the projection $W\oplus \squarepic\otimes V$ onto $W$ induces a short exact sequence
		\[
		0\to \squarepic\otimes V\to \cone(f)\to \coker(f)\to 0
		\]
		whose kernel is contractible.
		\item Given $U\xrightarrow{f} V\xrightarrow{g} W$, we have a distinguished triangle
		\[
		\cone(f)\to \cone(g\circ f)\to \cone(g).
		\]
		Indeed, if both $f,g$ are injective, then this follows from the previous example by noting
		\[
		0\to \coker(f)\to \coker(g\circ f)\to \coker(g)\to 0
		\]
		is short exact. In general, each map $g: V\to W$ is equivalent to an injection by replacing $W$ with $W\oplus \squarepic\otimes V$.
		\item Let $V[-1]\xrightarrow{\Phi} W$ be a bicomplex map. Then the quasi-isomorphism $\bullet\to \Lpic\otimes\revLpic$ induces a quasi-isomorphism between
	\[
	\cone(\Phi)_{red}=W\oplus V,\ \text{with }
	\del=\begin{pmatrix}
		\del_W & 0\\
		\phi & \del_V
	\end{pmatrix}, 
	\delbar=\begin{pmatrix}
		\delbar_W & 0\\
		\bar\phi & \delbar_V
	\end{pmatrix}
	\]
	and $\cone(\Phi)$. Here if we write $V[-1]=\del\otimes V+\delbar\otimes V+\del\delbar\otimes V$, then $\phi,\bar\phi$ are the restrictions of $\Phi$ onto $\del\otimes V$ and $\delbar\otimes V$ respectively. In particular
	\[
	V[-1]\xrightarrow{\Phi} W\to \cone(\Phi)_{red}
	\]
	is a distinguished triangle.
	\item $\tau_{\le k}V\to V\to \tau_{\ge k+1}V$ is distinguished. To see this, we argue that cokernel of $\tau_{\le k}V\to V$ is quasi-isomorphic to $\tau_{\ge k+1}V$. Indeed, we have a short exact sequence
	\[
	0\to K\to \coker(\tau_{\le k}V\to V)\to \tau_{\ge k+1}V\to 0
	\]
	where $K$ is the bicomplex with
	\[
	K^i=\begin{cases}
		V/\ker \del\delbar & i=k\\
		\ker\del\cap\ker\delbar/\ker\del\cap\ker\delbar\cap(\im\del+\im\delbar) & i=k+1\\
		\im \del\delbar & i=k+2\\
		0 &\text{otherwise}
	\end{cases}
	\]
	which is contractible.
	\end{enumerate}
\end{example}

\section{Homotopy theory of cbba}
A cbba is a (graded) commutative unital algebra object in the category of bicomplexes. Henceforth all bicomplexes and cbba's are concentrated in bidegrees $(i,j)$ with $i\ge 0, j\ge 0$.

An augmentation of a cbba $\mathcal{A}$ is a cbba map $\epsilon:\mathcal{A}\to \KK$, the augmentation ideal is denoted $\mathcal{A}^+$. If $\mathcal{A}$ is connected, i.e. $\mathcal{A}^0=\KK$, then $\mathcal{A}^+=\mathcal{A}^{>0}$. 

Throughout this section \textbf{all cbba's are augmented unless otherwise stated}.

\subsection{Linear Hirsch extension}
\begin{definition}\label{defn:hirsch-extension}
Let $(\mathcal{A},\del_\mathcal{A}, \delbar_{\mathcal{A}})$ be a cbba and $(V,\del_V,\delbar_V)$ a bicomplex. A \textbf{linear Hirsch extension} of $\mathcal{A}$ by $V$ is a pair of differentials $\del,\delbar$ on $\mathcal{A}\otimes \bigwedge V$ making it into a cbba so that 

\begin{enumerate}[(i)]
	\item the natural inclusion $\mathcal{A}\to \mathcal{A}\otimes \bigwedge V$
is a cbba map, and
\item the differentials on $V$ have the form
\[
\del=\del_V+\phi+\Theta,\  \delbar=\delbar_V+\bar\phi+\bar\Theta
\]
where $\phi,\bar\phi:V\to \mathcal{A}$ and $\Theta,\bar\Theta:V\to \mathcal{A}^+\otimes V$ are degree $+1$ linear maps.
\end{enumerate}
Note that the differentials on $\mathcal{A}\otimes \bigwedge V$ are uniquely determined by the conditions above.

An \textbf{algebra automorphism of $\mathcal{A}\otimes \bigwedge V$ relative to $\mathcal{A}$ and $V$} is an algebra automorphism that
\begin{enumerate}[(i)]
	\item restricts to identity on $\mathcal{A}$, and
	\item sends $x\in V$ to $x+H(x)$ where $H: V\to \mathcal{A}\otimes V$ is a degree $0$ linear map.
\end{enumerate}

Two Hirsch extensions $(\del,\delbar)$ and $(\del',\delbar')$ are isomorphic if they are differed by conjugating an algebra automorphism $\sigma$ relative to $\mathcal{A}$ and $V$, i.e. $(\del',\delbar')=\sigma^{-1}(\del,\delbar)\sigma$.
\end{definition}

Let us analyze the consequences of $\del^2=\delbar^2=(\del\pm\delbar)^2=0$. Starting with $\del^2=0$, for $x\in V$ we have:
\begin{align*}
	\del^2 x &= \del(\del_V x+\phi x+\Theta x)\\
	&=(\del_V+\phi+\Theta)\del_V x+\del_\mathcal{A}\phi x+ \del(\Theta x)\\
	&= \del_V^2 x+(\del_\mathcal{A}\phi x+\phi \del_Vx)+\Theta (\del_V x)+\del(\Theta x).
\end{align*}
All the terms, except the last one, have clear types: $\del_V^2 x\in V$ (which is zero), $(\phi \del_Vx+\del_\mathcal{A}\phi x)\in \mathcal{A}$ and $\Theta (\del_V x)\in\mathcal{A}\otimes V$. To understand the last term $\del(\Theta x)$, we note\footnote{Koszul sign convention understood, $(f\otimes g) (a\otimes b)=(-1)^{|g||a|}fa\otimes gb$.} $\del=\del_{\mathcal{A}}\otimes 1+1\otimes\del$ on $\mathcal{A}\otimes V$ and thus
\begin{align*}
	\del(\Theta x)&=(\del_{\mathcal{A}}\otimes 1+1\otimes\del)\Theta x\\
	&=(\del_\mathcal{A}\otimes 1)\Theta x+(1\otimes\del_V)\Theta x+(1\otimes \phi)\Theta x+(1\otimes\Theta)\Theta x\\
	&=\partial_{\mathcal{A}\otimes V}\Theta x+(1\otimes \phi)\Theta x+(1\otimes \Theta)\Theta x.
\end{align*}
Then, by analyzing types and degrees of the terms, we get that
\begin{enumerate}[(i)]
	\item $\del_V^2=0$ (which is given),
	\item $\del_\mathcal{A}\phi+\phi\del_V+(1\otimes \phi)\Theta=0$,
	\item $\del_{\mathcal{A}\otimes V}\Theta+\Theta\del_V+(1\otimes \Theta)\Theta=0$.
\end{enumerate}
To unwrap these equations, \textit{define} an operator $\del_\Theta$ on $Hom(V,\mathcal{A})$ by
\[
\del_\Theta\psi:=\partial_\mathcal{A}\psi-(-1)^{|\psi|}\psi\partial_V-(-1)^{|\psi|}(1\otimes\psi)\Theta\text{, for $\psi\in Hom(V,\mathcal{A})$.}
\]
Then the equation (ii) above becomes
\[
\del_\Theta\phi=0.
\]
And the equation (iii) above is the Maurer-Cartan equation for $\partial_\Theta$, showing that $\partial_\Theta$ is a deformation of the ordinary differential $\psi\mapsto \partial_\mathcal{A}\psi-(-1)^{|\psi|}\psi\partial_V$ on $Hom(V,\mathcal{A})$. In other words, (iii) is equivalent to
\[
\partial_\Theta^2=0.
\]
The same consideration applies to $(\delbar,\bar\Theta)$ and $(\del\pm\delbar,\Theta\pm\bar\Theta)$ as well, and thus we have:
\begin{enumerate}[(i)]
	\item $\del_\Theta^2=0$, $\del_\Theta\phi=0$;
	\item $\delbar_{\bar\Theta}^2=0$, $\delbar_{\bar\Theta}\bar\phi=0$;
	\item $\del_\Theta\bar\del_\Theta+\delbar_{\bar\Theta}\del_\Theta=0$, $\del_\Theta\bar\phi+\delbar_{\bar\Theta}\phi=0$.
\end{enumerate}
In other words, $Hom(V,\mathcal{A})$ equipped with $\del_\Theta,\delbar_\Theta$ becomes a bicomplex and that $(\phi,\bar\phi)$ define a bicomplex map $\Phi$ from $\revLpic$ into $Hom(V,\mathcal{A})$:
\[\Phi:=\left(\begin{tikzcd}[row sep=tiny, column sep=tiny]
	\bar\phi \ar[r]& {\del_\Theta \bar \phi=-\bar\del_{\bar\Theta}\phi} \\
	& {\phi}\ar[u]
\end{tikzcd}\right)\]
If we identify $Hom(\revLpic,Hom(V,\mathcal{A}))$ with $Hom(\revLpic\otimes V,\mathcal{A})$, then the actions of $\del_\Theta, \delbar_{\bar\Theta}$ translate to
\[
\left(\begin{tikzcd}[row sep=tiny, column sep=tiny]
	f & h \\
	& g
	\arrow[no head, from=1-1, to=1-2]
	\arrow[no head, from=2-2, to=1-2]
\end{tikzcd}\right)
\xrightarrow{\del_\Theta}
\left(\begin{tikzcd}[row sep=tiny, column sep=tiny]
	h-(-1)^{|h|}\del_\Theta f & (-1)^{|h|+1}\del_\Theta h \\
	& (-1)^{|h|+1}\del_\Theta g
	\arrow[no head, from=1-1, to=1-2]
	\arrow[no head, from=2-2, to=1-2]
\end{tikzcd}\right)
\]
and
\[
\left(\begin{tikzcd}[row sep=tiny, column sep=tiny]
	f & h \\
	& g
	\arrow[no head, from=1-1, to=1-2]
	\arrow[no head, from=2-2, to=1-2]
\end{tikzcd}\right)
\xrightarrow{\delbar_{\bar\Theta}}
\left(\begin{tikzcd}[row sep=tiny, column sep=tiny]
	(-1)^{|h|+1}\delbar_{\bar\Theta} f & (-1)^{|h|+1}\delbar_{\bar\Theta} h \\
	& -h-(-1)^{|h|}\delbar_{\bar\Theta} g
	\arrow[no head, from=1-1, to=1-2]
	\arrow[no head, from=2-2, to=1-2]
\end{tikzcd}\right)
\]
With these understood, we see that
\begin{enumerate}[(i)]
	\item $\del_\Theta$,$\delbar_{\bar\Theta}$ define differentials on $Hom(V[-1],\mathcal{A})$ making it into a bicomplex, and
	\item $\phi,\bar\phi$ define a map $\Phi: V[-1]\to \mathcal{A}$ satisfying $\del_\Theta\Phi=0, \delbar_\Theta\Phi=0$.
\end{enumerate}

\begin{definition}
	We call $\Theta,\bar\Theta: V\to \mathcal{A}\otimes V$ satisfying $\del_\Theta^2=\bar\del_{\bar\Theta}^2=\del_\Theta\bar\del_{\bar\Theta}+\del_\Theta\bar\del_{\bar\Theta}=0$ a \textbf{commuting pair of local systems} on $\mathcal{A}$. Note that $\Theta,\bar\Theta$ can be viewed as matrices with coefficients in $\mathcal{A}$, so we also call them \textbf{twisting coefficients}.
	
	We define \textbf{twisted homotopy} $[V[-1], \mathcal{A}]_{\Theta,\bar\Theta}$ to be the Bott-Chern cohomology of $Hom(V[-1],\mathcal{A})$ with respect to $\del_\Theta,\delbar_{\bar\Theta}$, i.e.
	\[
	[V[-1], \mathcal{A}]_{\Theta,\bar\Theta}:=\frac{\Phi \text{ s.t. } \del_\Theta\Phi=0,\delbar_{\bar\Theta}\Phi=0}{\Phi=\del_\Theta\delbar_{\bar\Theta}\Psi \text{ for some }\Psi}.
	\]
	Here the actions of $\del_{\Theta},\delbar_{\bar\Theta}$ are the ones described above. If $\Theta=\bar\Theta=0$, we recover the ordinary notion of homotopy \cite{stelzig2025pluripotential}.
\end{definition}

\begin{example}[Non-abelian Hodge correspondence \cite{simpson1992higgs}]
	Let $X$ be a complex manifold and $\mathcal{A}$ its de Rham algebra of complex differential forms. Let $V$ be a bicomplex (over $\CC$) concentrated in a single total degree (so $\del_V=\delbar_V=0$) and $\Theta,\bar\Theta: V\to\mathcal{A}\otimes  V$ a commuting pair of local systems, such that $\bar\Theta$ is the complex conjugate of $\Theta$. 
	On the one hand, $\nabla=d_\mathcal{A}-\Theta-\bar\Theta$ defines a (real) flat connection on $E=X\times V$. On the other hand, if we write
	\[
	\bar\Theta=\bar\Theta^{1,0}+\bar\Theta^{0,1},
	\]
	then $\bar\delta=\delbar_\mathcal{A}-\bar\Theta^{0,1}$ defines an operator $\mathcal{A}^0\otimes V\to \mathcal{A}^{0,1}\otimes V$ and $\varphi=\bar\Theta^{1,0}$ defines an operator $\mathcal{A}^0\otimes V\to \mathcal{A}^{1,0}\otimes V$. They satisfy
	\[
	\bar\delta^2=0,\ \bar\delta\varphi=0,\ \varphi\wedge\varphi=0.
	\]
	This means that $\bar\delta$ defines a holomorphic structure on $E$ and $\varphi$ defines a Higgs field.
\end{example}
\begin{question}
Is it possible to extend non-abelian Hodge correspondence to general compact complex manifolds along these lines?
\end{question}
\begin{remark}
	From the above example and the difficulty we shall see later in degree one when building minimal models, it appears to the author that the degree one information is more geometric than algebraic and perhaps should be handled by direct geometric methods. Poincar\'e defined fundamental group to be the "universal" monodromy group by means of local systems on a manifold. It is likely one can similarly consider commuting pairs of local systems on a complex manifold and directly analyze the corresponding symmetry group. This symmetry group should then be used to build the non-abelian Hodge correspondence for general compact manifolds.
\end{remark}

\begin{theorem}\label{hirsch-extension}
	A linear Hirsch extension of $\mathcal{A}$ by $V$ is determined, up to isomorphism, by
	\begin{enumerate}[(i)]
		\item a commuting pair of local systems $\Theta,\bar\Theta$ and
		\item a twisted homotopy class $[\Phi]\in[V[-1], \mathcal{A}]_{\Theta,\bar\Theta}$.
	\end{enumerate}
\end{theorem}

\begin{proof}
	We have seen that every Hirsch extension gives a commuting pair of local systems and a twisted homotopy class. Now let $\sigma$ be an algebra isomorphism relative to $\mathcal{A}$ and $V$ that sends $x\in V$ to $x+H(x)$ for $H:\:V\to \mathcal{A}$. Then $\sigma^{-1}$ sends $x$ to $x-H(x)$.
A straightforward computation shows that
\[
\sigma^{-1}\del\sigma=\del_V+\phi+\del_\Theta H+\Theta.
\]
Hence conjugation by $\sigma$ yields
\[
	\phi\mapsto \phi+\del_\Theta H,\quad \Theta\mapsto\Theta,
\]
and similar changes to $\bar\phi, \bar\Theta$. This proves that
	\[
	\Phi-\sigma^{-1}\Phi\sigma=\del_\Theta\delbar_{\bar\Theta} \left(\begin{tikzcd}[row sep=tiny, column sep=tiny]
	0 & H \\
	& 0
	\arrow[no head, from=1-1, to=1-2]
	\arrow[no head, from=2-2, to=1-2]
\end{tikzcd}\right).
	\]
	Therefore isomorphic Hirsch extensions give the same local systems $(\Theta,\bar\Theta)$ and homotopic $\Phi$. Conversely if $\Phi'$ and $\Phi$ are homotopic, that is
	\[
	\Phi-\Phi'=\del_\Theta\delbar_{\bar\Theta} \left(\begin{tikzcd}[row sep=tiny, column sep=tiny]
	f & h \\
	& g
	\arrow[no head, from=1-1, to=1-2]
	\arrow[no head, from=2-2, to=1-2]
\end{tikzcd}\right).
	\]
	Then the automorphism given by $H=h+\del_\Theta f+\delbar_{\bar\Theta} g$ takes $\Phi$ to $\Phi'$ by conjugation.	
\end{proof}

\begin{definition}
	In the situation of the above theorem, we say the linear Hirsch extension is structured by $[\Phi]$ and call $[\Phi]$ the \textbf{k-invariant} of the extension.
\end{definition}

\begin{proposition}\label{extension-property}
	Let $\mathcal{A}\otimes \bigwedge V$ be a linear Hirsch extension structured by $\Phi:V[-1]\to\mathcal{A}$ with local systems $\Theta,\bar\Theta$. Then a cbba map $f: \mathcal{A}\to \mathcal{C}$ extends to $\mathcal{A}\otimes \bigwedge V$ iff  $f\Phi$ is null-homotopic with respect to local systems $f\Theta,f\bar\Theta$.
\end{proposition}
\begin{proof}
	For simplicity, we denote $f\Theta,f\bar\Theta$ also by $\Theta,\bar\Theta$. If there is an extension $\widetilde{f}$ whose restriction to $V$ is $H:V\to \mathcal{C}$, then the commutativity $\del_\mathcal{C}\widetilde{f}=\widetilde{f}\del$ implies that for $x\in V$, $\del_\mathcal{C}Hx = H\del_V x+f\phi x+(1\otimes H)\Theta x$.
		That is, $f\phi=\del_\Theta H$. Similarly $f\bar\phi=\delbar_{\bar\Theta}H$. This proves that
		\[
		f\Phi=-\del_\Theta\delbar_{\bar\Theta} \left(\begin{tikzcd}[row sep=tiny, column sep=tiny]
	0 & H \\
	& 0
	\arrow[no head, from=1-1, to=1-2]
	\arrow[no head, from=2-2, to=1-2]
\end{tikzcd}\right)
		\]
		is null-homotopic. Conversely, if $f\Phi$ is null-homotopic then we can define an extension $H$ using a null-homotopy as in the proof of the previous theorem.
\end{proof}
\begin{corollary}
	Let $\mathcal{A}\otimes \bigwedge V$ be a linear Hirsch extension structured by $\Phi:V[-1]\to\mathcal{A}$ with local systems $\Theta,\bar\Theta$, and let $i: \mathcal{A}\to\mathcal{A}\otimes\bigwedge V$ be the canonical inclusion. Then the composition
	\[
	V[-1]\xrightarrow{\Phi}\mathcal{A}\xrightarrow{i}\mathcal{A}\otimes\bigwedge V
	\]
	is null-homotopic with respect to local systems $i\Theta,i\bar\Theta$.
\end{corollary}
\begin{proof}
	Apply the proposition to $i$ which extends to the identity map on $\mathcal{A}\otimes\bigwedge V$. 
\end{proof}
\begin{remark}
	From the proof of the previous proposition, there is in fact a canonical null-homotopy $\left(\begin{tikzcd}[row sep=tiny, column sep=tiny]
	0 & \iota_V \\
	& 0
	\arrow[no head, from=1-1, to=1-2]
	\arrow[no head, from=2-2, to=1-2]
\end{tikzcd}\right)
$ where $\iota_V$ is the inclusion of $V$ into $\mathcal{A}\otimes\bigwedge V$. This null-homotopy induces a map $V[-1]\to \cone(i)[-1]$ which further yields a map $V\to \cone(i)=(\mathcal{A}\otimes\bigwedge V)\oplus\mathcal{A}[1]$ given by $(\iota_V, \Phi)$. Here we identify $Hom(V[-1],\mathcal{A})=Hom(V,\mathcal{A}[1])$ since $Hom(\revLpic,\KK)=\Lpic$.

Recall that since $i$ is injective, we have a quasi-isomorphism $\cone(i)\xrightarrow{\simeq}\coker(i)$. The composition
\[
V\to\cone(i)\xrightarrow{\simeq}\coker(i)=V+\mathcal{A}^+\otimes V+\mathcal{A}\otimes \bigwedge^2 V+\dots
\]
is the caonoical inclusion.
\end{remark}



\begin{example}[tautological extension]\label{example:tautological extension}
\begin{enumerate}[(i)]
	\item Let $f: \mathcal{A}\to \mathcal{C}$ be a cbba map and let $W$ denote $\cone(f)=\mathcal{C}\oplus\mathcal{A}[1]$. Then there is a canonical map $\Phi: W[-1]\to \mathcal{A}$ tautologically given as follows. Write a typical element of $W$ as
	\[
	\left(c, \begin{tikzcd}[row sep=tiny, column sep=tiny]
	z \\
	x & y
	\arrow[no head, from=2-1, to=1-1]
	\arrow[no head, from=2-1, to=2-2]
\end{tikzcd}\right)
	\]
	and think of $c, x,y,z$ as maps from $W=\cone(f)$ into $\mathcal{C}$ and $\mathcal{A}$ accordingly.
	Then the map $\Phi$ is the mapping
	\[
	\left(\begin{tikzcd}[row sep=tiny, column sep=tiny]
	y & x \\
	& z
	\arrow[no head, from=1-1, to=1-2]
	\arrow[no head, from=2-2, to=1-2]
\end{tikzcd}\right).
	\]
	Now $f\Phi$ is canonically null-homotopic, indeed
	\[
	f\Phi=-\del\delbar\left(\begin{tikzcd}[row sep=tiny, column sep=tiny]
	0 & c \\
	& 0
	\arrow[no head, from=1-1, to=1-2]
	\arrow[no head, from=2-2, to=1-2]
\end{tikzcd}\right).
	\]
	Therefore $f$ extends to a map $\widetilde{f}:\mathcal{A}\otimes\bigwedge W\to \mathcal{C}$. This extension restricted to $W$ is simply the map $c$.
	
	Moreover, similar to the remark above, the null-homotopy induces a map
	\[
	W\to \cone(f)
	\]
	which is nothing but the identity map.
	\item More generally, let $V$ be any bicomplex equipped with a morphism $V\xrightarrow{g}\cone(f)$. Then we have an induced linear Hirsch extension $\mathcal{A}\otimes \bigwedge V$ of $\mathcal{A}$ structured by
	\[
	V[-1]\xrightarrow{g[-1]}\cone(f)[-1]\xrightarrow{\Phi}\mathcal{A}
	\]
	and an induced extension $\mathcal{A}\otimes \bigwedge V\to \mathcal{C}$.
\end{enumerate}
\end{example}

\subsection{Minimal algebra}
\begin{definition}
	A \textit{connected} cbba $\mathcal{M}$ is \textbf{minimal} if it is free as an algebra and minimal in the sense that $\del\delbar\mathcal{M}\subset\mathcal{M}^+\cdot\mathcal{M}^+$.
For a minimal cbba we define its \textbf{dual homotopy} bicomplex of indecomposables to be $\pi=\mm{M}/\mm{M}^+\cdot\mm{M}^+$.
\end{definition}
\begin{example}\label{dual-lie}
\begin{enumerate}[(i)]
	\item Let $\mathfrak{g}$ be a finite dimensional differential graded Lie algebra equipped with a vector space splitting $\mathfrak{g}=\mathfrak{g}_{1,0}\oplus \mathfrak{g}_{0,1}$ into two differential Lie subalgebras. Then the Cartan-Eilenberg algebra
\[
\bigwedge \mathfrak{g}^{\vee}=\bigwedge(\mathfrak{g}_{1,0}^{\vee}\oplus \mathfrak{g}_{0,1}^{\vee})
\]
becomes a bigraded algebra whose differential $d$ dual to the Lie bracket on $\mathfrak{g}$ also splits into $(1,0)$ and $(0,1)$ parts, making it into a cbba which is minimal for dimension reasons. It is clear that every free cbba (finitely) generated in degree one is isomorphic to one of these algebras. 

We will call a free cbba (not necessarily finitely) generated in degree one a \textbf{dual Lie algebra}.
\item Let $\mathfrak{a}\subset \mathfrak{g}$ be an ideal with quotient $\mathfrak{h}$ compatible with the previous splitting on $\mathfrak{g}$, i.e. $\mathfrak{a}=\mathfrak{a}\cap\mathfrak{g}_{1,0}+\mathfrak{a}\cap\mathfrak{g}_{0,1}$. Then $\bigwedge\mathfrak{g}^\vee$ is a linear Hirsch extension of $\bigwedge\mathfrak{h}^\vee$ by $\mathfrak{a}^\vee$. If further $0\to\mathfrak{a}\to \mathfrak{g}\to \mathfrak{h}\to 0$ is a central extension, then the extension is structured by an untwisted homotopy class.
\item One can now speak of \textit{solvable} resp. \textit{nilpotent} dual Lie algebras using (ii). On the dual side, this means the dual Lie algebra is obtained by a finite chain of linear Hirsch extensions structured by twisted resp. untwisted homotopy classes.
\end{enumerate}
\end{example}

We now analyze the structure of a connected minimal cbba $\mathcal{M}$. For simplicity, let us first assume that $\mathcal{M}^1=0$. Let $\mathcal{M}_k$ be the subalgebra of $\mathcal{M}$ generated by
\[\tau_{\le k}\pi=\pi^{\le k}\oplus (\im\del+\im\delbar)\cap \pi^{k+1}.\]
Then
\begin{enumerate}[(i)]
	\item $\mathcal{M}_0=\mathcal{M}_1=\KK$.
	\item Each $\mathcal{M}_k$ is closed under differentials $\del,\delbar$. To see this, $\tau_{\le k}\pi$ consists of all generators of $\mathcal{M}$ up to degree $k$ and those generators in degree $(k+1)$ which are hit (mod. decomposables) by differentials. 
	\begin{itemize}
		\item For a generator in degree $\le k-1$, its differentials are in degree $k$ which must be contained in $\mathcal{M}_k$ for dimension reason.
		\item For a generator in degree $k$, either its differentials are decomposable which are contained in $\mathcal{M}_k$ for dimension reason, or its differentials modulo decomposables hit some generator in degree $(k+1)$, but those generators are included in $\tau_{\le k}\pi$.
		\item Finally for a generator of degree $(k+1)$ in $\tau_{\le k}\pi$, its differentials are decomposable and thus contained in
	\[
	\mathcal{M}^1\cdot \mathcal{M}^{k+1}+\mathcal{M}^2\cdot \mathcal{M}^{k}+\cdots
	\]
	All the terms are contained in $\mathcal{M}_k$, except possibly for $\mathcal{M}^1\cdot \mathcal{M}^{k+1}$. But $\mathcal{M}^1=0$ by our simply-connectedness assumption.
	\end{itemize}
	\item From (ii) we also see that, for $k\ge 3$, $\mathcal{M}_k$ is a linear Hirsch extension of $\mathcal{M}_{k-1}$ by $H^k\pi$. For $k=2$, the degree $3$ generators in $H^2\pi$ might hit a term in $\mathcal{M}^2\cdot\mathcal{M}^2$ which could involve a quadratic expression of generators in $H^2\pi$. So the extension $\mathcal{M}_1\subset\mathcal{M}_2$ may not be linear.
\end{enumerate}

For a general minimal algebra, the situation is even  worse because the subalgebras $\mathcal{M}_k$ are not necessarily closed under differentials. This motivates the definitions below.

\begin{definition}
A connected minimal cbba $\mathcal{M}$ is said to be \textbf{autonomous} if all $\mathcal{M}_k$'s are closed under differentials. The autonomous condition is equivalent to that the \textit{differential} subalgebra generated by $\mathcal{M}^{\le k}$ is $\mathcal{M}_k$.
\end{definition}

\begin{definition}\label{defn:solvable-nilpotent}
A minimal cbba $\mathcal{M}$ is said to be \textbf{solvable} if it has a filtration by sub-cbba's $\KK=F_0\mathcal{M}\subset F_1\mathcal{M}\subset F_2\mathcal{M}\subset\dots\subset \bigcup_k F_k\mathcal{M}=\mathcal{M}$ such that
\begin{enumerate}[(i)]
	\item each $F_k \mathcal{M}$ is a linear Hirsch extension over $F_{k-1}\mathcal{M}$, and
	\item $\mathcal{M}$ is degree-wise exhausted by finitely many $F_k\mathcal{M}$'s.
\end{enumerate}
If moreover the linear extensions are structured by ordinary untwisted homotopy classes, then we say $\mathcal{M}$ is \textbf{nilpotent}. The filtration $F_\bullet$ is called a \textbf{solvable series} resp. \textbf{nilpotent series}.

By dropping the assumption (ii), we have notions of \textit{generalized} solvable and \textit{generalized} nilpotent algebras.
\end{definition}

The following is a useful observation.
\begin{lemma}\label{cocycle-exists}
	Let $\mathcal{M}$ be a generalized solvable algebra filtered by $F_k\mathcal{M}$ as in \Cref{defn:solvable-nilpotent}. Let $i$ be the first positive degree for which $\mathcal{M}^i\neq 0$. Then $H^i_A\mathcal{M}\neq 0$.
\end{lemma}
\begin{proof}
	 Consider the first $k$ for which $(F_k\mathcal{M})^i=F_k\mathcal{M}\cap \mathcal{M}^i\neq 0$, we show that $\ker\del\delbar\cap (F_k\mathcal{M})^i\neq 0$ and thus they define non-zero cohomology classes since $\mathcal{M}^{i-1}=0$. Assume otherwise, then $(F_k\mathcal{M})^i$ injects by $\del\delbar$ into decomposables in $(F_{k-1}\mathcal{M})^{i+2}$. But this decomposable subspace
	 \[
	 (F_{k-1}\mathcal{M})^1\cdot (F_{k-1}\mathcal{M})^{i+1}+(F_{k-1}\mathcal{M})^2\cdot (F_{k-1}\mathcal{M})^i+\dots
	 \]
	 must be zero by our assumptions on $i$ and $k$.
\end{proof}

\begin{definition}An extension $\mathcal{A}\otimes \bigwedge V$ of $\mathcal{A}$ by $V$ differed from being a linear Hirsch extension by a quadratic term $V\to \bigwedge ^2 V$ is called a \textbf{quadratic-linear} Hirsch extension.
\end{definition}

\begin{theorem}[structure of minimal cbba]
	 An autonomous minimal cbba $\mathcal{M}$ is determined up to isomorphism by:
	 \begin{enumerate}[(i)]
	 	\item the minimal algebra $\mathcal{M}_2$; and
	 	\item an inductive sequence of twisted homotopy classes:
	 	\[
	 	\Phi_k\in \left[V_k[-1],\mathcal{M}_{k-1}\right]_{\Theta^k,\bar\Theta^k},\  k=3,4,5,\dots
	 	\]
	 	describing $\mathcal{M}_k$ as a linear Hirsch extension of $\mathcal{M}_{k-1}$ by $V_k=H^k\pi$.
	 \end{enumerate}
\end{theorem}

\subsection{Quadratic-linear Hirsch extension}
Let $\mathcal{A}\otimes \bigwedge V$ be a quadratic-linear Hirsch extension whose differentials are determined on $V$ by
\[
\del =\del_V+\phi+\Theta, \ \delbar=\delbar_V+\bar\phi+\bar\Theta.
\]
This is similar to the linear case, except that now $\del_V=\del_V^L+\del_V^Q$ has a linear component $\del_V^L: V\to V$ and an extra quadratic component $\del_V^Q: V\to \bigwedge^2 V$. Then just like in the linear case, analyzing the types of the terms in the equation $\del^2=0$ shows that
\begin{enumerate}[(i)]
	\item $\del_V^2=0$, in $\bigwedge V$;
	\item $\del_\mathcal{A}\phi+\phi\del_V^L+(1\otimes \phi)\Theta=0$, in $\mathcal{A}^+$;
	\item $\del_{\mathcal{A}\otimes V}^L\Theta+\Theta\del_V^L+(1\otimes \Theta)\Theta+\phi\del_V^Q=0$, in $\mathcal{A}^+\otimes V$;
	\item $\Theta\del_V^Q+(1\otimes \del_V^Q)\Theta=0$, in $\mathcal{A}^+\otimes \bigwedge^2 V$.
\end{enumerate}

The equation (i) implies that $V$ is a dual differential graded Lie algebra: if $V$ is finite dimensional then $V^\vee$ carries a differential dual to $\del_V^L$ and a skew-symmetric binary operator $[-,-]$ dual to $\del_V^Q$, the equation $\del_V^2=0$ is equivalent to that $V^\vee$ is a differential graded Lie algebra. With this understood, the equation (iv) can be interpreted as follows: $\mathcal{A}\otimes V^\vee$ is now equipped with the structure of a differential graded Lie algebra over $(\mathcal{A},\del_V)$; and $\Theta$, when identified as a map
\[
V^\vee\to \mathcal{A}\otimes V^\vee,
\]
has to be a Lie homomorphism.
We may define $\del_\Theta^L$ as before so that the equation (ii) reads $\del_\Theta^L\phi=0$. Now the equation (iii) shows that $(\del_\Theta^L)^2$ is no longer zero, but obstructed by a "curvature" term $\phi\del_V^Q$. If we identify $\phi$ as an element in $\mathcal{A}\otimes V^\vee$, then $\phi\del_V^Q=0$ is equivalent to that
\[
[\phi,-]=0.
\]
In general the equation (iii) says that $\del_\Theta^L+[\phi,-]$ is a differential operator.

To proceed, we impose a flatness condition.
\begin{definition}
	A quadratic-linear Hirsch extension is \textbf{flat} if $(\del_\Theta^L)^2=0$.
\end{definition}
Then the classification of flat quadratic-linear Hirsch extensions is similar to the linear case, except that $V$ is more structured now and $\Theta, \phi$ etc. all respect the extra structures.

\begin{theorem}\label{flat-extension}
	A flat quadratic-linear Hirsch extension of $\mathcal{A}$ by $V$ is determined, up to isomorphism, by
	\begin{enumerate}[(i)]
		\item a commuting pair of local systems $\Theta,\bar\Theta$ and
		\item a twisted homotopy class $[\Phi]\in\left[V[-1], \mathcal{A}\right]_{\Theta,\bar\Theta}$,
	\end{enumerate}
	in which $\Theta,\bar\Theta$ and $\Phi$ (and homotopies) all respect the corresponding dual differential graded Lie structures on $V$.
\end{theorem}

If we further allow the differential $\del_V$ to have all higher order terms, then $V$ is a dual Lie-infinity algebra and the general structure of Hirsch extensions is better phrased in terms of deformation theory of Lie-infinity algebras on the dual side. 

The quadratic (especially semi-simple) part of the differentials in degrees above one, the curvature term, as well as the non-closedness of $\mathcal{M}_k$ are new features of the theory of minimal cbba's, when compared to the theory of minimal cdga's. We will pursue these elsewhere.

\subsection{Homotopy}
We define a notion of homotopy convenient for dealing with local coefficients. The difference between the notion herein and that in \cite{stelzig2025pluripotential} is the difference between left and right homotopy, which coincide in the situation we concern since all cbba's are fibrant while linear Hirsch extensions are cofibrations. We will need the following lemma in order to construct cylinder objects for defining of homotopy.
\begin{lemma}[contracting homotopy]
\begin{enumerate}[(i)]
	\item Let $s, \bar{s}:W\to W$ be a pair of \textbf{contracting homotopies}, i.e. $\del s+s\del=id$, $\del\bar s+\bar s\del=0$, $\delbar \bar s+\bar s\delbar=id$ and $\delbar s+s\delbar=0$. Then $W$ is contractible, i.e. $H_A W=0$ and $H_{BC}W=0$.
	\item Let $\mathcal{A}$ be a cbba and $\Theta,\bar\Theta$ a pair of commuting local systems on $\mathcal{A}$ with coefficients in $V$. Let $s,\bar s:\mathcal{A}\to \mathcal{A}$ be a pair of contracting homotopies, and let $\mathcal{O}_{s,\bar s}$ denote the subspace of $\mathcal{A}$ whose left-multiplication operations graded-commutes\footnote{note that $s,\bar s$ are of degree $-1$.} with both $s,\bar s$.
	
	If the images of $\Theta,\bar\Theta$ lie in $\mathcal{O}_{s,\bar s}\otimes V$, then $Hom(V,\mathcal{A})$ is contractible with respect to $\del_\Theta, \delbar_{\bar\Theta}$.
\end{enumerate}
\end{lemma}
\begin{proof}
	For (i), to see that $H_{BC}=0$, let $x\in \ker\del\cap\ker\delbar$, then $\del sx=x$ and $\delbar sx=0$. Meanwhile $s x=(\delbar \bar s+\bar s\delbar) s x=\delbar \bar s  s x$ and thus $x=\del s x=\del\delbar (\bar s s x)$. The proof for $H_A=0$ is similar. Now (ii) follows from (i) by noting that the anti-commuting condition implies that $s,\bar s$ induce contracting homotopies on $Hom(V,\mathcal{A})$ with respect to $\del_\Theta, \delbar_{\bar\Theta}$.
\end{proof}

\begin{corollary}\label{contractibility}
	$\mathcal{A}\otimes \bigwedge \square\xrightarrow{\text{mod $\square$}} \mathcal{A}$ is a quasi-isomorphism for all local coefficients in $\mathcal{A}$ and $\text{ideal}(\square)$ is contractible for all coefficients in $\mathcal{A}$.
\end{corollary}
\begin{proof}
	There are contracting homotopies on $\square$ which can be extended to $\mathcal{A}$-derivations on $\mathcal{A}\otimes \bigwedge \square$. So they define contracting homotopies on $\text{ideal}(\square)$ that graded-commute with left multiplications by $\mathcal{A}$.
\end{proof}

\begin{remark}[cf. \cite{stelzig2025pluripotential}]
	This can be used to construct \textit{integration operators} on $\mathcal{A}\otimes \bigwedge \square$ sensitive to local coefficients as homotopies between identity map and the composition $ev_0:\mathcal{A}\otimes \bigwedge \square\xrightarrow{\text{mod $\square$}} \mathcal{A}\hookrightarrow \mathcal{A}\otimes \bigwedge \square$. Inductively (degree-wise) one constructs operators $s,\bar s$ so that $\del_\Theta s+s\del_\Theta=id-ev_0$, $\delbar_{\bar\Theta} s+s\delbar_{\bar\Theta}=0$ and similar for $\bar s$.
\end{remark}

Now we are ready to construct cylinder objects and define homotopy. Let $\mathcal{A}$ be a \textit{successive} linear Hirsch extension over $\mathcal{B}$ with local coefficients all in $\mathcal{B}$. This means, if we denote by $\{x_\alpha\}$ a homogeneous basis of \textit{new} generators added into $\mathcal{B}$, then (recall \Cref{defn:hirsch-extension}) we can write differentials as
\begin{align*}
	\del x_\alpha&=\eta_{\alpha\beta} x_\alpha+\phi_\alpha+\theta_{\alpha\beta}x_\beta\\
	\delbar x_\alpha&=\bar\eta_{\alpha\beta} x_\alpha+\bar\phi_\alpha+\bar\theta_{\alpha\beta}x_\beta
\end{align*}
in which $\eta_{\alpha\beta}, \bar\eta_{\alpha\beta}\in \KK$, $\phi_\alpha,\bar\phi_\alpha$ are polynomials in \textit{previously added} generators with coefficients in $\mathcal{B}$ and twisting coefficients $\theta_{\alpha\beta},\bar\theta_{\alpha\beta}\in \mathcal{B}$.

Let us write $\mathcal{A}=\mathcal{B}(x_\alpha)$ and $\mathcal{A}\otimes_\mathcal{B}\mathcal{A}=\mathcal{B}(x_\alpha,y_\alpha)$, then we \textit{define}
\[
\mathcal{A}^\square:=\mathcal{B}(x_\alpha, y_\alpha;\varphi_\alpha,\del\varphi_\alpha,\delbar\varphi_\alpha)
\]
in which $\varphi_\alpha, \del\varphi_\alpha,\delbar\varphi_\alpha$ are new generators (with prescribed differentials indicated by their names) respectively in degree $(1,1), (1,0), (0,1)$ less than
\[
\delta_\alpha=\delta(x_\alpha,y_\alpha)=x_\alpha-y_\alpha.
\]

\begin{proposition}[cylinder object]\label{cylinder}
	There exist differentials on $\mathcal{A}^\square$ extending that on $\mathcal{B}(x_\alpha,y_\alpha)$ such that
	\begin{enumerate}[(i)]
		\item $\del_\Theta\delbar_{\bar\Theta} \varphi_\alpha=\delta_\alpha+\text{terms in previous generators}$, and
		\item ideal $(\varphi_\alpha,\del\varphi_\alpha,\delbar\varphi_\alpha,\del\delbar\varphi_\alpha)$ is contractible for all coefficients in $\mathcal{B}$.
	\end{enumerate}
\end{proposition}
\begin{proof}
	The proof goes by induction on new generators using contractibility of the ideal. For simplicity, we write $x,y$ for new generators, and $x',y'$ for previous generators. Then we can write
	\begin{align*}
		&\del_\Theta x=\phi(x'),\  \delbar_{\bar\Theta} x=\bar\phi(x'),\\
		&\del_\Theta y=\phi(y'),\  \delbar_{\bar\Theta} y=\bar\phi(y'),
	\end{align*}
	and thus
	\[
	\del_\Theta\delta=\phi(x')-\phi(y'), \ \delbar_{\bar\Theta} \delta=\bar\phi(x')-\bar\phi(y').
	\]
	Note that $\phi(x')-\phi(y')$ and $\bar\phi(x')-\bar\phi(y')$ belong to the previous ideal
	\[
	(\varphi',\del\varphi',\delbar\varphi',\del\delbar\varphi')=(\varphi',\del\varphi',\delbar\varphi',\delta')
	\]
	where the equality follows from (i) by induction hypothesis. Now we claim there are $\epsilon,\zeta,\eta$ from the previous ideal satisfying the equations
	\begin{align*}
		-\del_\Theta\delbar_{\bar\Theta}\epsilon &=\del_\Theta\delbar_{\bar\Theta}\delta,\\
		-\del_\Theta\delbar_{\bar\Theta}\eta &=\del_\Theta\delta+\del_\Theta\epsilon,\\
		\del_\Theta\delbar_{\bar\Theta}\zeta &=\delbar_{\bar\Theta}\delta+\delbar_{\bar\Theta}\epsilon.
	\end{align*}
	Indeed the first equation has a solution because the right hand side is contained in the previous ideal and also in $\ker\del_\Theta\cap\ker\delbar_{\bar\Theta}$. Since the previous ideal is $\del_\Theta, \delbar_{\bar\Theta}$-contractible, its Bott-Chern cohomology vanishes and therefore $\epsilon$ can be found. Similarly we can find $\zeta,\eta$. Then we define
	\[
	\del_\Theta\delbar_{\bar\Theta}\varphi:=\delta+\epsilon+\del_\Theta\zeta+\delbar_{\bar\Theta}\eta.
	\]
	The new ideal is contractible by \Cref{connectivity} using the isomorphism
	\[
	\mathcal{A}^\square=\mathcal{A}\otimes\text{algebra}(\varphi,\del\varphi,\delbar\varphi,\del\delbar\varphi).
	\]
\end{proof}
\begin{remark}\label{kinv-homotopy}
\begin{enumerate}[(i)]
	\item There are plenty of rooms for choosing $\varepsilon, \zeta,\eta$, which comes handy later in the proof of \Cref{obstruction theory}.
	\item Note that $\mathcal{A}^\square$ is a successive linear Hirsch extension over $\mathcal{A}\otimes_\mathcal{B}\mathcal{A}$. By construction, the image of its inductive k-invariant
\[
\begin{tikzcd}[row sep=tiny, column sep=tiny]
	{\delbar\varphi} \\
	{\varphi} & {\del\varphi}
	\arrow[from=2-1, to=1-1]
	\arrow[from=2-1, to=2-2]
\end{tikzcd}[-1]\to \mathcal{A}\otimes_\mathcal{B}\mathcal{A}(\varphi',\del\varphi',\delbar\varphi')
\]
lies in the ideal $(\varphi',\del\varphi',\delbar\varphi',\delta',\delta)$. Using the quasi-isomorphism $\bullet\to \revLpic\otimes\Lpic$, this k-invariant can be identified with a twisted Bott-Chern cohomology class, which is nothing but the twisted Bott-Chern class of $(\delta+\text{previous terms})$.
\end{enumerate}
\end{remark}

\begin{definition}[homotopy] Two cbba maps from $\mathcal{A}$ into $\mathcal{C}$ extending $\mathcal{B}\xrightarrow{\nu}\mathcal{C}$ are \textbf{homotopic} (rel $\nu$) if the combined map from $\mathcal{A}\otimes_{\mathcal{B}}\mathcal{A}$ into $\mathcal{C}$ extends to $\mathcal{A}^\square$. The equivalence classes of maps (rel $\nu$) up to homotopy will be denoted\footnote{to distinguish from bicomplex homotopy.} by $hHom(\mathcal{A},\mathcal{C};\nu)$.
\end{definition}
\begin{lemma}
	Homotopic cbba maps are homotopic as bicomplex maps.
\end{lemma}
\begin{proof}
	This follows from that the two natural inclusions $x,y:\mathcal{A}\to\mathcal{A}^\square$ become isomorphisms when post-composed with the quasi-isomorphism $\mathcal{A}^\square\to \mathcal{A}$ quotienting out the ideal $(\varphi,\del\varphi,\delbar\varphi,\delta)$.
\end{proof}

\begin{proposition}
	Homotopy is an equivalence relation and homotopies can be added.
\end{proposition}
\begin{proof}
	The proof goes by induction using \Cref{extension-property}. We inductively extend the inclusion
	\[
	\mathcal{B}(x,z)\to \mathcal{B}(x,y,z;\varphi_{xy},\del\varphi_{xy},\delbar\varphi_{xy};\varphi_{yz},\del\varphi_{yz},\delbar\varphi_{yz})
	\]
	to $\mathcal{B}(x,z;\varphi_{xz},\del\varphi_{xz},\delbar\varphi_{xz})$. The inductive obstruction is the push-forward of the k-invariant (\Cref{kinv-homotopy}), which has image in ideal $(\varphi_{xy},\del\varphi_{xy},\delbar\varphi_{xy},\delta_{xy};\varphi_{yz},\del\varphi_{yz},\delbar\varphi_{yz},\delta_{yz})$. Since this ideal is contractible we can proceed without obstruction. This proves homotopies can be added and transitivity. Symmetry follows from the involution swapping the roles of $x$ and $y$ on $\mathcal{A}^\square$.
\end{proof}

\subsection{Obstruction theory}
\begin{theorem}
	Let $\mathcal{A}$ be a successive linear Hirsch extension of $\mathcal{B}$. The bicomplex of new generators will be called the \textbf{dual homotopy} of $(\mathcal{A},\mathcal{B})$ and denoted by $\pi(\mathcal{A},\mathcal{B})$. Then
	\begin{enumerate}[(i)]
		\item (Existence) the inductive construction of extending a map $\nu: \mathcal{B}\to \mathcal{C}$ to $\mathcal{A}$ is obstructed by a sequence of twisted homotopy class in $[\pi(\mathcal{A},\mathcal{B})[-1],\mathcal{C}]$; and
		\item (Uniqueness) the inductive construction of extending a homotopy between maps $\mathcal{A}\to\mathcal{C}$ rel $\nu$is obstructed by a sequence of twisted homotopy class in $[\pi(\mathcal{A},\mathcal{B}),\mathcal{C}]$.
	\end{enumerate}
\end{theorem}
\begin{proof}
	The existence part is \Cref{extension-property}. The uniqueness part follows from our definition of homotopy by \Cref{kinv-homotopy}.
\end{proof}
In short, we can say extending a map amounts to solving the structure equations
	\[
	\del_\Theta x =\phi, \delbar_{\bar\Theta}x=\bar\phi;
	\]
and extending a homotopy amounts to solving the structure equation
	\[
	\del_\Theta\delbar_{\bar\Theta} \varphi=\delta+\epsilon+\del_\Theta\zeta+\delbar_{\bar\Theta}\eta.
	\]

This type of equation-solving argument further proves that:
\begin{theorem}\label{obstruction theory} Let $\mathcal{A}$ be a successive linear Hirsch extension over $\mathcal{B}$.
Given $\nu:\mathcal{B}\to\mathcal{C}$, $p:\mathcal{C}'\to \mathcal{C}$ and a lifting $\nu':\mathcal{B}\to \mathcal{C}'$ such that $p\nu'=\nu$.
\[
	\begin{tikzcd}
		 \mathcal{B}\ar[r,"\nu'"]\ar[d,"i"']&\mathcal{C}'\ar[d,"p"]\\
		 \mathcal{A}\ar[r,"f"']\ar[ur,dashed,"f'"]& \mathcal{C}
	\end{tikzcd}
	\]
\begin{enumerate}[(i)]
	\item The inductive construction of lifting $\mathcal{A}\xrightarrow{f}\mathcal{C}$ to $\mathcal{A}\xrightarrow{f'}\mathcal{C}'$ \underline{up to homotopy}, meaning that $pf'$ is homotopic to $f$, is obstructed by a sequence of twisted homotopy classes in $[\pi(\mathcal{A},\mathcal{B}),\cone(p)]$.
	\item Moreover, if $p$ is a surjective quasi-isomorphism for coefficients in $\pi(\mathcal{A},\mathcal{B})$, then we can find a lifting $f'$ such that $pf'=f$.
\end{enumerate}
\end{theorem}
\begin{proof} We are given $x\to \mathcal{C}$, and need to find $y'\to \mathcal{C}'$ and $\varphi\to \mathcal{C}$ solving the equations (suppressing local coefficients for simplicity)
\begin{enumerate}[(i)]
	\item in $\mathcal{C}'$, $\del y'=\phi'$, $\del\bar y'=\bar\phi'$;
	\item in $\mathcal{C}$, $\del\delbar \varphi=\delta+\epsilon+\del\zeta+\delbar\eta$.
\end{enumerate}
Now define an element $C$ of $\cone(p)$ by
	$C:=\left(x+\epsilon+\del\zeta+\delbar\eta,
		\begin{tikzcd}[row sep=tiny, column sep=tiny]
	-\bar\phi' \\
	\delbar\phi' & \phi'
	\arrow[no head, from=2-1, to=1-1]
	\arrow[no head, from=2-1, to=2-2]
\end{tikzcd}\right)$.
	Note that $\varepsilon,\zeta,\eta$ are constructed in $\mathcal{C}$ by previous steps as in \Cref{cylinder}, such that
	\[
	\del C=\delbar C=0.
	\]
	Therefore $C$ defines a class in $[\pi,\cone(p)]$. Recall that $C$ is null-homotopic iff $C=\del\delbar B$ for some $B$.
	
	If a lifting $y'$ and a homotopy $\varphi$ can be found then we see that 
	$B:=\left(\varphi,\begin{tikzcd}[row sep=tiny, column sep=tiny]
	0 \\
	-y' & 0
	\arrow[no head, from=2-1, to=1-1]
	\arrow[no head, from=2-1, to=2-2]
\end{tikzcd}\right)$
	satisfies
	\[
	\del\delbar B=C
	\]
	by the equations (i) and (ii) above. Conversely if $C=\del\delbar B$ for some $B=\left(\varphi,\begin{tikzcd}[row sep=tiny, column sep=tiny]
	c' \\
	-y' & b'
	\arrow[no head, from=2-1, to=1-1]
	\arrow[no head, from=2-1, to=2-2]
\end{tikzcd}\right)$, then we can modify $\varepsilon,\zeta,\eta$ to $\varepsilon+y, \zeta-b,\eta-c$ such that $C=\del\delbar\left(\varphi,\begin{tikzcd}[row sep=tiny, column sep=tiny]
	0 \\
	-y' & 0
	\arrow[no head, from=2-1, to=1-1]
	\arrow[no head, from=2-1, to=2-2]
\end{tikzcd}\right)$, thus giving a lifting $y'$ and a homotopy $\varphi$. This completes the proof of the first assertion.

Now if $p$ is surjective, we can find a candidate $x'\to \mathcal{C}'$ by directly lifting $x\to\mathcal{C}$. Then
\[
\left(\begin{tikzcd}[row sep=tiny, column sep=tiny]
	(\delbar x'-\bar\phi') \\
	(\del\delbar x'-\delbar\bar\phi') & (\phi'-\del x')
	\arrow[no head, from=2-1, to=1-1]
	\arrow[no head, from=2-1, to=2-2]
\end{tikzcd}\right)
\]
defines a cocycle of $Hom(\pi,\ker\pi[1])\simeq Hom(\pi,\cone(p))$ which is null-homologous. Then we use a homotopy to modify $x'$ so that $\del x'=\phi'$, $\delbar x'=\bar\phi'$.
\end{proof}
\begin{corollary}
	If $p$ is a quasi-isomorphism with respect to local coefficients in $\pi(\mathcal{A},\mathcal{B})$, then it induces a bijection
	\[
	hHom(\mathcal{A},\mathcal{C}';\nu')\to hHom(\mathcal{A},\mathcal{C};\nu).
	\]
\end{corollary}
\begin{proof}
	The obstructions to existence and uniqueness all vanish.
\end{proof}

\subsection{Algebraic fibration}
In this subsection, we treat the inductive step in proving the uniqueness of minimal models. The situation we encounter is a pair $\mathcal{B}\subset \mathcal{E}$ of connected cbba's in which $\mathcal{E}$ is a free module over $\mathcal{B}$ so that $\mathcal{E}=\mathcal{B}\otimes \mathcal{F}$. Through the isomorphism $\mathcal{F}=\mathcal{E}/\text{ideal }\mathcal{B}^+$, $\mathcal{F}$ inherits the structure of a cbba. The pair $(\mathcal{E},\mathcal{B})$ will be called an \textit{algebraic fibration} and $\mathcal{E}$, $\mathcal{B}$, $\mathcal{F}$ are called the \textit{total space}, the \textit{base} and the \textit{fiber} respectively.
The prototype of an algebraic fibration is a linear Hirsch extension.
%
\begin{proposition}
If $\mathcal{F}\xrightarrow{f}\mathcal{F}'$ is a quasi-isomorphism (ordinary coefficients), then $\mathcal{E}\xrightarrow{f}\mathcal{E}'$ is a quasi-isomorphism for all coefficients in $\mathcal{B}$.
\end{proposition}
\begin{proof}
	Notice that $\mathcal{B}^{\ge k}\otimes \mathcal{F}$ are sub-bicomplexes for all twisting coefficients in $\mathcal{B}$, and the induced differentials on the relative quotients $\mathcal{B}^k\otimes \mathcal{F}$ are trivial on the $\mathcal{B}$-factor.
\end{proof}
\begin{proposition}
	 Suppose $\mathcal{F}$, $\mathcal{F}'$ are generalized solvable (recall \Cref{defn:solvable-nilpotent}), and $ \mathcal{E}\xrightarrow{f}\mathcal{E}'$ is a quasi-isomorphism for all twisting coefficients\footnote{not all coefficients are needed, see proof and remark below.}. Then $f$ is an actual cbba isomorphism on fibers $\mathcal{F}$, $\mathcal{F}'$. Therefore $\mathcal{E}$, $\mathcal{E}'$ are isomorphic as cbba's.
\end{proposition}
\begin{proof}
The proof is based on the observation that new generators in the fibers are reflected in relative cohomology groups, and we use obstruction theory to identify new generators in the fibers degree-by-degree.

	Let $i$ be the first positive degree for which $\mathcal{F}^i\neq 0$, the $\mathcal{E}$-differentials take $\mathcal{F}^i$ into $\mathcal{F}^{i+1}+\mathcal{B}^{i+1}+\mathcal{B}^1\otimes\mathcal{F}^{i}$. Therefore $U_i:=\mathcal{F}^i+(\im\del_\mathcal{F}+\im\delbar_\mathcal{F})^{i+1}+(\im\del_\mathcal{F}\delbar_\mathcal{F})^{i+2}$ defines (tautological) local coefficients on $\mathcal{B}$. We claim that
	\[
	V_i:=\left(\ker\del_{\mathcal{F}}\delbar_{\mathcal{F}}\right)^i+\left(\ker\del_{\mathcal{F}}\cap\ker\delbar_{\mathcal{F}}\cap(\im\del_{\mathcal{F}}+\im\delbar_{\mathcal{F}})\right)^{i+1}
	\]
	defines sub-coefficients and consequently a linear Hirsch extension $\mathcal{B}\otimes\bigwedge V_i$ over $\mathcal{B}$. Indeed for dimension reasons $H^i U_i=V_i$ and thus we can write $U_i$ as a direct sum of $V_i$ and
	\[
\begin{tikzcd}[row sep=small, column sep=small]
 	\del y\ar[r] & \del\delbar y\\
 	y \ar[u]\ar[r]& \delbar y\ar[u]
\end{tikzcd}
	\]
For $x$ standing for a basis of $(V_i)^i$, we can write $\del x=\del_V x+\phi x+\Theta x+\Lambda y$. Then by analyzing types of the terms from the equation $\del^2=0$, we see that $\Lambda.\del y=0$ which in turn implies $\Lambda=0$; similarly we have $\bar\Lambda=0$. This proves the claim.
 
	Now consider $V_i'\subset\mathcal{F}'$ similarly defined and the map $V_i\xrightarrow{f} V_i'$, we argue that this map is injective. First notice that since $V_i$ defines local coefficients, so does $W_i:=\im V_i\subset V_i'$. Now consider the diagram:
	\[
	\begin{tikzcd}
		\mathcal{B}\arrow[r, hook]\arrow[d,hook] & \mathcal{E}\ar[d,"f"]\\
		\mathcal{B}\otimes\bigwedge W_i \arrow[r, hook]\ar[ur,dashed]&\mathcal{E}'
	\end{tikzcd}
	\]
By obstruction theory (\Cref{obstruction theory}) the canonical embedding $\mathcal{B}\otimes\bigwedge W_i\hookrightarrow\mathcal{E}'$ can be lifted to a map into $\mathcal{E}$ up to homotopy (rel $\mathcal{B}$). Then consider the (commutative) diagram
\[
	\begin{tikzcd}
		\mathcal{B}\arrow[r, hook]\arrow[d,hook] & \mathcal{E}\ar[d,"f"]\\
		\mathcal{B}\otimes\bigwedge V_i \arrow[r]\arrow[ur,hook]&\mathcal{E}'
	\end{tikzcd}
	\]
By obstruction theory (\Cref{obstruction theory}) the map $\mathcal{B}\otimes\bigwedge V_i\hookrightarrow\mathcal{E}$ is unique up to homotopy. Therefore it must be homotopic (rel $\mathcal{B}$) to the composition $\mathcal{B}\otimes\bigwedge V_i\xrightarrow{f} \mathcal{B}\otimes\bigwedge W_i$ followed by the map $\mathcal{B}\otimes\bigwedge W_i\to \mathcal{E}$ previously obtained. Thus reducing modulo ideal $\mathcal{B}^+$, we have a diagram
\[
\begin{tikzcd}
	\bigwedge V_i\arrow[r,hook]\ar[d,"f"] &\mathcal{F}\ar[d,"f"]\\
	\bigwedge W_i \ar[ur]\arrow[r, hook]&\mathcal{F}'
\end{tikzcd}
\]
that commutes up to homotopy. Applying the cohomology functor we obtain an actual commutative diagram
\[
\begin{tikzcd}
	V_i \ar[d]\arrow[r,equal]& H^i V_i\ar[r,"\simeq"]\ar[d,"f"] &H^i\mathcal{F}\ar[d,"f"]\\
	W_i \arrow[r,equal]& H^i W_i \ar[ur]\ar[r]&H^i\mathcal{F}'
\end{tikzcd}
\]
This forces $ V_i\cong W_i$, i.e. $V_i\xrightarrow{f} V_i'$ is injective.

On the other hand, let $j$ be the first positive degree for which $(\mathcal{F}')^j\neq 0$, we can show that $V_j\xrightarrow{f} V_j'$ is surjective by considering the lifting problem
\[
\begin{tikzcd}
	\mathcal{B}\arrow[r, hook]\arrow[d,hook] & \mathcal{E}\ar[d,"f"]\\
	\mathcal{B}\otimes\bigwedge V_j' \arrow[r, hook]\ar[ur,dashed]&\mathcal{E}'
\end{tikzcd}
\]
Since $\mathcal{F},\mathcal{F}'$ are assumed to be generalized solvable, by \Cref{cocycle-exists} $V_i, V_j'$ are non-zero, which in turn implies that $i=j$ and $V_i=V_i'$. Now iteratively replace $\mathcal{B}$ by $\mathcal{B}\otimes \bigwedge V_i$ to conclude that $\mathcal{F}\simeq \mathcal{F}'$.
\end{proof}
\begin{remark}
	It suffices to require $\mathcal{E}$, $\mathcal{E}'$ to be quasi-isomorphic for all twisting coefficients appeared in the proof, we call these \textbf{action coefficients}.
\end{remark}

\begin{corollary}\label{uniqueness-solvable} A map between two \underline{generalized solvable} algebras that induces a quasi-isomorphism for all action coefficients is an actual cbba isomorphism.
\end{corollary}
\begin{corollary}\label{uniqueness-nilpotent} A map between two \underline{generalized nilpotent} algebras that induces a quasi-isomorphism (ordinary coefficients) is an actual cbba isomorphism.
\end{corollary}
\begin{corollary}
	A map between two autonomous minimal algebras which induces an isomorphism up to degree $2$ and a quasi-isomorphism for action coefficients in $\mathcal{M}_2$ on $H^{\ge 3}\pi$, is a cbba isomorphism.
\end{corollary}

%

\subsection{Minimal model}
Let $\mathcal{A}$ be an Aeppli connected ($H_A^0=\KK$) and Aeppli simply-connected ($H_A^1=0$) cbba, and we now try to construct a minimal model for it. We would like to inductively construct $\varphi_k:\mathcal{M}_{k}\to \mathcal{A}$ so that $\varphi_k$ is $(k+1)$-connected, that is $H_A^{\le k}\varphi_k$ is surjective and $H_{BC}^{\le k+2}\varphi_k$ is injective. 

To start with, set $\mathcal{M}_0=\mathcal{M}_1=\KK$ and $\varphi_0=\varphi_1$ the inclusion of ground field which is clearly $2$-connected.

Now for $k\ge 2$, suppose we have constructed $\varphi_{k-1}:\mathcal{M}_{k-1}\to \mathcal{A}$, we need to construct $\mathcal{M}_k$ and $\varphi_k$. Consider $\cone(\varphi_{k-1})$, which is $(k-1)$-connected by assumption, and set $V=H^k\cone(\varphi_{k-1})$ and \underline{choose} a quasi-isomorphism $V\xrightarrow{\simeq}\tau_{\le k}\cone(\varphi_{k-1})$.

Define a Hirsch extension $\mathcal{M}_{k}:=\mathcal{M}_{k-1}\otimes\bigwedge V$ structured by the composition
\[
\Phi: V[-1]\xrightarrow{\simeq}\tau_{\le k}\cone(\varphi_{k-1})[-1]\to\cone(\varphi_{k-1})[-1]\to \mathcal{M}_{k-1}.
\]
Then $\varphi_{k-1}$ extends to a map $\varphi_k:\mathcal{M}_k\to \mathcal{A}$ because the obstruction to extension vanishes by design; indeed (recall \Cref{example:tautological extension} ) the composition
\[
\cone(\varphi_{k-1})[-1]\to \mathcal{M}_{k-1}\xrightarrow{\varphi_{k-1}}\mathcal{A}
\]
is null-homotopic. Let us analyze the connectivity of $\varphi_k$.


\begin{lemma}\label{inductive step A}
Assume that $\mathcal{M}_{k-1}^{+}$ is concentrated in degrees $\ge 2$, i.e. $\mathcal{M}_{k-1}^1=0$. Then $\varphi_{k}$ is $k$-connected and $H_{BC}^k\left(\cone(\varphi_k)\right)=0$.
\end{lemma}
\begin{proof}
	Let $i:\mathcal{M}_{k-1}\hookrightarrow\mathcal{M}_k$ be the inclusion. Then recall from \Cref{example:tautological extension} we have an induced commutative diagram (induced by the canonical null-homotopies):
	\[
	\begin{tikzcd}
		V\ar[r,"\alpha"]\ar[dr,"\beta"'] & \cone(i)\ar[d,"\gamma"]\ar[r,"\simeq","\delta"'] & \coker(i)\\
		 & \cone(\varphi_{k-1}) & 
	\end{tikzcd}
	\]
	where $\beta$ is the composition $V\xrightarrow{\simeq}\tau_{\le k}\cone(\varphi_{k-1})\hookrightarrow\cone(\varphi_{k-1})$, $\gamma$ is induced by $\varphi_k$ and $\delta$ is the quasi-isomorphism in \Cref{special-triangle}. That $\beta=\gamma\circ\alpha$ yields a distinguished triangle
	\[
	\cone(\alpha)\to \cone(\beta)\to\cone(\gamma).
	\]

	\begin{enumerate}[(1)]
		\item Since $\delta$ is a quasi-isomorphism, $\cone(\alpha)$ is quasi-isomorphic to $\cone(j)\simeq \coker(j)$ where $j=\delta\circ\alpha$ is the canonical inclusion
		\[
		j: V\hookrightarrow V+\mathcal{M}_{k-1}^+\otimes V+\mathcal{M}_{k-1}\otimes\bigwedge^2 V+\dots
		\]
		\item From \Cref{special-triangle}, $\cone(\beta)$ is quasi-isomorphic to $\tau_{\ge k+1}\cone(\varphi_{k-1})$.
		\item Since $i\circ\varphi_k=\varphi_{k-1}$, we have a distinguished triangle
	\[
	\cone(i)\xrightarrow{\gamma}\cone(\varphi_{k-1})\to\cone(\varphi_k).
	\]
	So $\cone(\gamma)$ is quasi-isomorphic to $\cone(\varphi_k)$.
	\end{enumerate}
Applying the long exact sequence \cite[Corollary 11.2]{stelzig2025pluripotential} to $\cone(\beta)\to\cone(\gamma)\to\cone(\alpha)[1]$ and using the quasi-isomorphisms above, we get exact sequences
\begin{enumerate}[(1)]
	\item $H_{BC}^i\left(\tau_{\ge k+1}\cone(\varphi_{k-1})\right)\to H^i_{BC}\cone(\varphi_k)\to H^i_{BC}\coker(j)[1]$
	\item $H_{A}^i\left(\tau_{\ge k+1}\cone(\varphi_{k-1})\right)\to H^i_{A}\cone(\varphi_k)\to H^{i}_{A}\coker(j)[1]$
\end{enumerate}

Since $\tau_{\ge k+1}\cone(\varphi_{k-1})$ is $k$-connected, its Aeppli and Bott-Chern cohomologies vanish up to degree $k$. It remains to analyze the cohomology of $\coker(j)[1]$.
\begin{enumerate}[(1)]
	\item (Bott-Chern cohomology) From \cite[Theorem 1.1]{stelzig2025pluripotential}
	\[
	H^{p,q}_{BC}\coker j[1]=\left[\bullet[p,q],\coker j [1]\right]=\left[\revLpic[p+1,q+1],\coker j\right]
	\]
	where $\bullet[p,q]$ is the bicomplex $\bullet$ located in bidegree $(p,q)$, the second equality follows from the hom-tensor adjunction and that $Hom(\Lpic,\KK)=\revLpic$, and the notation $\revLpic[p+1,q+1]$ stands for the bicomplex $\revLpic$ in which the conner is in degree $(p+1,q+1)$.
	
	Now observe that $\coker j=\mathcal{M}_{k-1}^+\otimes V+\mathcal{M}_{k-1}\otimes\bigwedge^2 V+\mathcal{M}_{k-1}\otimes\bigwedge^3 V+\cdots$ is concentrated in degrees $\ge k+2$. But $\revLpic[p+1,q+1]$ is generated in (total) degree $p+q+1$. Therefore $H^{p,q}_{BC}\coker j[1]=0$ as long as $p+q<k+1$, i.e. $H^i_{BC}\coker j[1]=0$ for $i\le k$. It follows that $H^i_{BC}\cone(\varphi_k)=0$ for $i\le k$.
	\item (Aeppli cohomology) Similarly (recall \Cref{Aeppli of shift}) from \cite[Theorem 1.1]{stelzig2025pluripotential}
	\[
	H_A^{p,q}\coker j[1]=H_{BC}^{p+1,q+1}\coker j.
	\]
	So for dimension reasons $H_A^i\coker j[1]=0$ for $i<k$ and thus $H_A^i\cone(\varphi_{k})=0$ for $i<k$. This implies that $\varphi_k$ is $k$-connected, thus completing the proof.
\end{enumerate}


%
\end{proof}

\begin{remark}\label{inductive remark A}
	We point out that condition $\mathcal{M}_{k-1}^1=0$ can be relaxed to that $\mathcal{M}_{k-1}^+$ is 1-connected. Indeed, $\mathcal{M}_{k-1}^+\otimes V$ is a sub-bicomplex of $\coker j$ whose relative quotient is concentrated in degrees $\ge k+2$. Moreover the induced differentials from $\mathcal{M}_k$ on $\mathcal{M}_{k-1}^+\otimes V$ are exactly the tensor product differentials.
	
	Since $\mathcal{M}_{k-1}^+$ is 1-connected and $V$ is $(k-1)$-connected, we see that $\mathcal{M}_{k-1}^+\otimes V$ is $(k+1)$-connected. This proves that $H^{i}_{BC}(\mathcal{M}_{k-1}^+\otimes V)=0$ for $i<k+2$ and consequently from the short exact sequence
	\[
	0\to \mathcal{M}_{k-1}^+\otimes V\to \coker j \to \mathcal{M}_{k-1}\otimes\bigwedge^{\ge 2} V\to 0
	\]
	we conclude $H^i_{BC}\coker j=0$ for $i<k+2$.
	
	Similarly the above short exact sequence induces an exact sequence by applying the functor $\left[\revLpic[p+1,q+1],-\right]$. Since $\mathcal{M}_{k-1}^+\otimes V$ is quasi-isomorphic to a bicomplex concentrated in degrees $\ge k+2$ and $\mathcal{M}_{k-1}\otimes\bigwedge^{\ge 2} V$ is concentrated in degrees $\ge k+2$, we conclude $\left[\revLpic[p+1,q+1],\coker j\right]=0$ for $p+q<k+1$ as needed.
 
\end{remark}

\begin{lemma}\label{inductive step B}
	Under the same assumptions as the previous lemma, if moreover $$H_{BC}^k\cone(\varphi_{k-1})=0,$$ then $\varphi_k$ is $(k+1)$-connected.
\end{lemma}
\begin{proof}
	Following the proof of the previous lemma, it suffices to prove $H^{k+2}_{BC}\coker j=0$. Now the differentials on
	\[
	\coker j=\mathcal{M}_{k-1}^+\otimes V+\mathcal{M}_{k-1}\otimes\bigwedge^2 V+\mathcal{M}_{k-1}\otimes\bigwedge^3 V+\cdots
	\]
	are inherited from $\mathcal{M}_{k}=\mathcal{M}_{k-1}\otimes \bigwedge V$ whose differentials preserve $\mathcal{M}_{k-1}^{\ge p}\otimes\bigwedge^{\le q} V$.
	
	So $K=\mathcal{M}_{k-1}^+\otimes V+\mathcal{M}_{k-1}\otimes \bigwedge^2 V$ is a sub-bicomplex of $\coker j$ and we have a short exact sequence
	\[
	0\to K\to \coker j\to \mathcal{M}_{k-1}\otimes \bigwedge^{\ge 3} V\to 0
	\]
	which then yields a short exact sequence
	\[
	H^{k+2}_{BC} K\to H^{k+2}_{BC}\coker j\to H^{k+2}_{BC}(\mathcal{M}_{k-1}\otimes \bigwedge^{\ge 3} V).
	\]
	Notice that $H^{k+2}_{BC}(\mathcal{M}_{k-1}\otimes \bigwedge^{\ge 3} V)=0$ for dimension reasons, therefore it remains to show $H^{k+2}_{BC} K=0$. Similarly we can further reduce to proving that $H^{k+2}_{BC}(\mathcal{M}_{k-1}^i\otimes V)=0$ for $i\ge 2$ and $H^{k+2}_{BC}(\mathcal{M}_{k-1}^i\otimes\bigwedge^2 V)=0$ for $i\ge 0$.  Most of them vanish for dimension reasons, we only need to show that $H^{k+2}_{BC}(\mathcal{M}_{k-1}^2\otimes V)=0$ and $H^{k+2}_{BC}(\bigwedge^2 V)=0$.
	\begin{enumerate}[(1)]
		\item The induced differentials vanish on the $\mathcal{M}_{k-1}$-factor in $\mathcal{M}_{k-1}^2\otimes V$, and thus 
		\[
		H^{k+2}_{BC}(\mathcal{M}_{k-1}^2\otimes V)=\mathcal{M}_{k-1}^2\otimes H^{k}_{BC}V=0.
		\]
		\item For dimension reasons $H^{k+2}_{BC}(\bigwedge^2 V)=0$ for $k>2$. If $k=2$, then $H^4_{BC}(\bigwedge^2 V)\subset H^4_{BC}(V\otimes V)=0$ by K\"unneth (\cite[Theorem 1.35]{stelzig2025pluripotential}). This completes the proof.
	\end{enumerate}
\end{proof}
\begin{remark}\label{inductive remark B}
	Similar to the previous remark, the condition $\mathcal{M}_{k-1}^1=0$ can be relaxed to that $\mathcal{M}_{k-1}^+$ is 1-connected. Indeed in this case $H_{BC}^{k+2}(\mathcal{M}_{k-1}^+\otimes V)=0$ by K\"unneth \cite[Theorem 1.35]{stelzig2025pluripotential}.
\end{remark}

\begin{theorem}[nilpotent model]\label{nilpotent-model}
	Let $\mathcal{A}$ be an Aeppli connected ($H^0_A=\KK$) and Aeppli simply-connected ($H^1_A=0$) cbba.
	\begin{enumerate}[(i)]
		\item (Existence) There exists a minimal nilpotent cbba $\mathcal{M}$ with a cbba quasi-isomorphism $\varphi: \mathcal{M}\to \mathcal{A}$.
		\item (Uniqueness) If $\mathcal{M}\xrightarrow{\varphi}\mathcal{A}$ and $\mathcal{M}'\xrightarrow{\varphi'}\mathcal{A}$ are two such, then there is a cbba isomorphism $\sigma:\mathcal{M}'\to \mathcal{M}$ and a homotopy between $\varphi\sigma$ and $\varphi'$; the isomorphism $\sigma$ is determined by these conditions up to homotopy.
	\end{enumerate}
	We call $\mathcal{M}$ the \textbf{nilpotent model} of $\mathcal{A}$.
\end{theorem}
\begin{proof}
	Take $\mathcal{M}_0=\mathcal{M}_1=\KK$ with $\varphi_1$ the inclusion of ground field. Suppose $\varphi_{k-1}:\mathcal{M}_{k-1}\to \mathcal{A}$ has been constructed. Then applying the construction yields an extension $\mathcal{M}_k^{(1)}$ and a map $\varphi_k^{(1)}:\mathcal{M}_{k}^{(1)}\to \mathcal{A}$ which is $k$-connected and $H_{BC}^k\cone(\varphi_k^{(1)})=0$. If $\varphi_k^{(1)}$ is already $(k+1)$-connected, then we take $\mathcal{M}_k=\mathcal{M}_k^{(1)}$ and $\varphi_k=\varphi_k^{(1)}$. If not, then we apply the construction one more time to obtain an extension $\mathcal{M}_k^{(2)}$ and a map $\varphi_k^{(2)}:\mathcal{M}_k^{(2)}\to \mathcal{A}$. This map must be $(k+1)$-connected by the previous lemma, so we take $\mathcal{M}_k=\mathcal{M}_k^{(2)}$ and $\varphi_k=\varphi_k^{(2)}$. It is clear that if $\mathcal{M}_{k-1}^1=0$ then $\mathcal{M}_k^1=0$ as well.

	Finally we obtain a nilpotent minimal algebra $\mathcal{M}=\bigcup_k \mathcal{M}_k$ with a quasi-isomorphism $\varphi=\lim \varphi_k:\mathcal{M}\to \mathcal{A}$.
	
	Uniqueness follows from \Cref{uniqueness-nilpotent} and obstruction theory.
\end{proof}

\begin{remark}\label{nilpotent-implies-auto}
	\begin{enumerate}[(i)]
		\item Note that $\mathcal{M}_k$ is constructed from $\mathcal{M}_{k-1}$ by adding generators in degree $k,k+1$; the degree $k+1$ generators are hit by the differentials of degree $k$ generators (mod. decomposables). Therefore $\mathcal{M}_k$ is the sub-cbba of $\mathcal{M}$ generated by $\mathcal{M}^{\le k}$ and in particular $\mathcal{M}$ is autonomous.
		\item The nilpotent series of $\mathcal{M}$ given by $\mathcal{M}_k^{(n)}$ is canonical. Indeed, denote $\mathcal{M}_k^{(0)}=\mathcal{M}_{k-1}$ and assume that an automorphism of $\mathcal{M}$ preserves $\mathcal{M}_k^{(n-1)}$, then it preserves the cokernel of $\mathcal{M}_k^{(n-1)}\hookrightarrow\mathcal{M}$ and thus preserves $H^k$ of cokernel, which consists of new generators in $\mathcal{M}_k^{(n)}$. We call this nilpotent series the \textbf{lower central series}.
		\item The construction of $\mathcal{M}$ becomes completely formal above the cohomological dimension of $\mathcal{A}$.
	\end{enumerate}
\end{remark}

\begin{remark}
	Our construction fails to produce a connected model without the Aeppli simply connectedness assumption, similar to \cite[Example 3.13]{stelzig2025pluripotential}.
\end{remark}

%

\begin{theorem}[relative model]
	Let $f: \mathcal{B}\to \mathcal{A}$ be cbba map. Suppose $\mathcal{B}$ is Aeppli connected and Aeppli simply-connected, and assume $f$ is $2$-connected.
	\begin{enumerate}[(i)]
		\item (Existence) There exists an algebraic fibration $\mathcal{E}$ over $\mathcal{B}$ with generalized nilpotent fiber, and a map $\varphi: \mathcal{E}\to \mathcal{A}$ extending $f$ so that $\varphi$ is a quasi-isomorphism.
		\item (Uniqueness) Any two such fibrations $\mathcal{E},\mathcal{E}'$ are isomorphism; the isomorphisms $\sigma$ are well-defined up to homotopy by the condition that $\varphi\sigma$ is homotopic to $\varphi'$ (rel $\mathcal{B}$).
	\end{enumerate}
\end{theorem}
\begin{proof}
	The construction and proof of \Cref{nilpotent-model} in fact goes through in the relative case, by the virtue of of \Cref{inductive remark A} and \Cref{inductive remark B}. Note that $\mathcal{B}$ being Aeppli connected and Aeppli simply-connected implies that the inclusion of scalars $\KK\to \mathcal{B}$ is $2$-connected and thus $\mathcal{B}^+\cong\mathcal{B}/\KK$ is $1$-connected.	
	
	Take $\mathcal{E}_1=\mathcal{B}$ and $\varphi_1=f$ and we inductively construct $(k+1)$-connected maps $\varphi_k:\mathcal{E}_k\to \mathcal{A}$ as before. We leave it to the reader to check that the property $\mathcal{E}_k^+$ is 1-connected is persistent through the inductive construction by noting that each term of
	\[
	\mathcal{E}_k^+=\mathcal{E}_{k-1}^+ +V+\mathcal{E}_{k-1}^+\otimes V+\mathcal{E}_{k-1}\otimes\bigwedge^{\ge 2}V
	\]
	is 1-connected.
\end{proof}

\section{Applications}
\subsection{Automorphism}

\begin{lemma}
	An automorphism of a connected minimal algebra $\mathcal{M}$ is unipotent iff it acts unipotently on either one of the following: dual homotopy $\pi=\mathcal{M}/\mathcal{M}^+\cdot\mathcal{M}^+$, $\del\delbar$-spherical cohomology $\Sigma_{\del\delbar}=\ker\del\delbar/(\mathcal{M}^+\cdot\mathcal{M}^+\cap\ker\del\delbar)$ and $\del\delbar$-cohomology $H_{\del\delbar}=\ker\del\delbar/\im\del\delbar$.
\end{lemma}
\begin{proof}
	This follows from the exact sequences
	\begin{enumerate}[(i)]
		\item $0\to \mathcal{M}^+\cdot\mathcal{M}^+\to \mathcal{M}\to \pi\to 0$,
		\item $0\to \Sigma_{\del\delbar}\to \pi\to \mathcal{M}/(\ker \del\delbar+\mathcal{M}^+\cdot\mathcal{M}^+)\to 0$,
		\item $0\to \mathcal{M}^+\cdot\mathcal{M}^+/\im \del\delbar\to H_{\del\delbar}\to \Sigma_{\del\delbar}\to 0$.
	\end{enumerate}
	The action of an automorphism on $\mathcal{M}^+\cdot\mathcal{M}^+$ is determined by its action on lower degree generators. Note that the action on $\mathcal{M}/(\ker \del\delbar+\mathcal{M}^+\cdot\mathcal{M}^+)$ is determined by the action on $\mathcal{M}/\ker \del\delbar\cong \im \del\delbar\subset \mm{M}^+\cdot\mm{M}^+$.
\end{proof}

Now we observe that the action of an automorphism on $\im\del+\im\delbar$ is determined by its action on lower degree subspaces, and thus we can take $\im\del+\im\delbar$ out from the groups in the previous lemma. This proves:
\begin{proposition}\label{unipotency}
	An automorphism of a connected minimal algebra $\mathcal{M}$ is unipotent iff it acts unipotently on either one of the following:
	\begin{align*}
		\pi_A & =\frac{\mathcal{M}}{\im\del+\im\delbar+\mathcal{M}^+\cdot\mathcal{M}^+},\\
		\Sigma_A& =\frac{\ker\del\delbar}{\mathcal{M}^+\cdot\mathcal{M}^+\cap\ker\del\delbar+\im\del+\im\delbar},\\
		H_A &= \frac{\ker\del\delbar}{\im\del+\im\delbar}.
	\end{align*}
\end{proposition}
\begin{remark}
	These groups are related by $H_A\twoheadrightarrow \Sigma_A\hookrightarrow \pi_A$. The group $\pi_A$ is the Aeppli cohomology of $\pi$.
\end{remark}
\begin{corollary}
	An automorphism of $\mathcal{M}$ homotopic to identity is unipotent. We call these \textbf{inner automorphisms}.
\end{corollary}
\begin{proof}
	Such an automorphism acts trivially on Aeppli cohomology.
\end{proof}

To summarize, we have
\begin{theorem}
	The following symmetry groups are differed only by normal unipotent subgroups.
	\begin{enumerate}[(i)]
		\item all automorphisms, denoted by $\aut\mathcal{M}$;
		\item outer automorphsims (i.e. all aut. modulo inner aut.), denoted by $h\aut\mathcal{M}$;
		\item $H_A$ automorphisms\footnote{we mean those induced from $\aut\mathcal{M}$, same below.};
		\item $\pi_A$ automorphisms;
		\item $\Sigma_A$ automorphisms.
	\end{enumerate}
	If $\pi_A$ is finite dimensional, then all of these symmetry groups are algebraic matrix groups with the same reductive part.
\end{theorem}
\begin{proof}
	Note that $\pi_A$ being finite dimensional implies that $\pi$ is finite dimensional and therefore $\mathcal{M}$ is finitely generated. Then an automorphism is determined by its action on generators subject to algebraic conditions imposed by commuting with differentials.
\end{proof}
\begin{remark}
The Lie algebra of the inner automorphism group should be that of inner derivations, i.e. derivations of the form $[\del,[\delbar,i]]$ in which $i$ is a bidegree $(-1,-1)$ derivation, see \cite{infinitesimal}. The Lie algebra of the outer automorphism group should be identified with Andr\'e-Quillen cohomology of $\mathcal{M}$, see \cite{block-lazarev}. We do not pursue these here.
\end{remark}

\begin{corollary}\label{finite-stage-aut}
	Let $\mathcal{M}$ be a connected nilpotent cbba, finitely generated in each degree. Assume $H_A$ is finite dimensional, then $h\aut\mathcal{M}$ is an algebraic matrix group.
\end{corollary}
\begin{proof}
Let $\mathcal{M}_k\subset\mathcal{M}$ a nilpotent series.
	We apply obstruction theory to argue that the natural restriction map induces an isomorphism $$hHom(\mathcal{M},\mathcal{M})\cong hHom(\mathcal{M}_k,\mathcal{M}_k)$$ for $k$ sufficiently large. Then it follows that $h\aut\mathcal{M}$ is isomorphic to $h\aut\mathcal{M}_k$; the latter is an algebraic matrix group by the previous theorem.
	
	Now given any self-map $f$ of $\mathcal{M}_k$, the obstructions for extending $\varphi_k f:\mathcal{M}_k\to \mathcal{M}$ to a self-map of $\mathcal{M}$ lie in $[\pi(\mathcal{M},\mathcal{M}_k)[-1],\mathcal{M}]$. Notice that $\pi(\mathcal{M},\mathcal{M}_k)$ is highly-connected, while $\mathcal{M}$ is quasi-isomorphic to $\tau_{\le n}\mathcal{M}$ for some fixed $n$ since $H_A\mathcal{M}$ is finite dimensional. So for $k$ large enough (so that $\pi(\mathcal{M},\mathcal{M}_k)$ is at least $(n+1)$-connected), the obstructions vanish for dimension reasons. Similarly, the obstructions to uniqueness lie in $[\pi(\mathcal{M},\mathcal{M}_k),\mathcal{M}]$ which vanishes for $k$ large enough.
\end{proof}

\begin{example}\label{example:aut-groups} Let $p(x_1,\dots,x_n)$ be a homogeneous polynomial and we consider
	\[
	\mathcal{M}=\bigwedge(x_1,\dots,x_n; y,z)
	\]
	in which $|x_i|=(1,1)$ with $\del x_i=\delbar x_i=0$ and $\del y=\delbar z=p(x_1,\dots,x_n)$. Let $\sigma$ be an automorphism of $\mathcal{M}$, then $\sigma$ is determined by its action on generators, and can be presented by a triple $A\in \GL_n,\lambda\in \GG_m,\mu\in \GG_m$ as $\sigma(x,y,z)=(Ax, \lambda y,\mu z)$. Since $\sigma$ commutes with $\del,\delbar$, we deduce that $\lambda=\mu$ and
	\[
	p(Ax)=\lambda p(x).
	\]
	Hence $\aut\mathcal{M}$ is the extension
	\[
	1\to \aut(p)\to \aut\mathcal{M}\to \GG_m\to 1.
	\]
	For instance, if $p=x_1^2+\dots+x_n^2$ then $\aut\mathcal{M}$ is the orthogonal similitude group.
	
	The (cdga) Sullivan model of $\mathcal{M}$ is $\mathcal{M}_S=\bigwedge(x_1,\dots,x_n; w)$ with $dx_i=0$ and $dw=p(x)$ whose (cdga) automorphism group is canonically isomorphic to $\aut\mathcal{M}$.
\end{example}
\begin{remark}
	Ignoring products with multiplicative groups, all connected reductive groups appear in this example. See \cite[Theorem 6.1]{infinitesimal}.
\end{remark}

 One would love to compare, in general, the (cdga) automorphism of the Sullivan model of a minimal cbba to its (cbba) automorphism. Let $\mathcal{M}$ be a minimal cbba and $\mathcal{M}_S\to \mathcal{M}$ its Sullivan model. Then we have natural map
\[
\text{outer automorphism of $\mathcal{M}$}\to \text{outer automorphism of $\mathcal{M}_S$}
\]
by obstruction theory. The corresponding map between Lie algebras should be
\[
\text{cbba Andr\'e-Quillen cohomology}\to \text{cdga Andr\'e-Quillen cohomology}
\]
induced by the canonical forgetful functor.
\begin{problem}
	Describe the map on Andr\'e-Quillen cohomology induced by forgetful functor. Note this however does not determine the map between automorphism groups since they are not necessarily connected.
\end{problem}
\begin{problem}
	Analyze the general structure of kernel and cokernel.
\end{problem}


\begin{example}[$\CC\PP^n$]\label{example:CPn}
Let $\mathcal{M}=\bigwedge (x, y,\del y,\delbar y)$ with $|x|=(1,1)$, $\del x=\delbar x=0$ and $\del\delbar y=x^{n+1}$. Then $\aut\mathcal{M}$ is $\GG_m\times \GG_a$ and the (cdga) automorphism of its Sullivan model is $\GG_m$.
\end{example}
\subsection{Formality}

\begin{definition}[cf.\cite{milivojevic2024bigraded}]
	A minimal cbba $\mathcal{M}$ is (strongly) \textbf{formal} if it is quasi-isomorphic, as cbba, to its Bott-Chern cohomology.
\end{definition}

\begin{definition}
	By the \textbf{grading automorphism} of a bigraded vector space $V^{\bullet,\bullet}$, we mean the action of $\GG_m\times \GG_m$ in which $(\alpha,\beta)$ acts on $V^{p,q}$ through multiplication by $\alpha^p\beta^q$. By a \textbf{lifting} of grading automorphism on $\mathcal{M}$, we mean a homomorphism
\[
\GG_m\times \GG_m\to \aut\mathcal{M}
\]
which passes to the grading automorphisms on \textit{both} Bott-Chern and Aeppli cohomology of $\mathcal{M}$.
\end{definition}

\begin{theorem}
	Let $\mathcal{M}$ be a finitely generated nilpotent minimal cbba with $H_A^1=0$. Then $\mathcal{M}$ is formal iff it carries a lifting of grading automorphism.
\end{theorem}
\begin{proof}
	If we have a quasi-isomorphism $\mathcal{M}\to H_{BC}(\mathcal{M})$, then this map is necessarily surjective and all relevant cohomology groups coincide. Then obstruction theory implies that all automorphisms lift, i.e. 
\[
\rho: \aut\mathcal{M}\to \aut H
\]
is surjective. Since $\GG_m\times\GG_m$ is reductive and $\ker \rho$ is unipotent, the grading automorphism $\GG_m\times\GG_m\to \aut H$ lifts to $\aut\mathcal{M}$ (recall Levi decomposition).

Conversely assume we have a lifting $\sigma: \GG_m\times\GG_m\to \aut\mathcal{M}$ of grading automorphism. Then standard linear representation theory of $\GG_m\times\GG_m$ shows that $\mathcal{M}$ splits into a direct sum of bigraded weight subspaces. Since $\sigma$ is further a cbba automorphism, all the weight subspaces are sub-bicomplexes and weight is additive with respect to multiplication.

We claim that $\mathcal{M}$ satisfies $\del\delbar$-lemma (so all relevant cohomology groups coincide) and that weight $\ge$ degree\footnote{$(i,j)\ge (i',j')$ if $i\ge i'$ and $j\ge j'$.}. Granted these for the moment, write
\[
\mathcal{M}=\{\text{weight$=$degree}\}\oplus \{\text{weight$>$degree}\}=S\oplus I.
\]
We can deduce that
\begin{itemize}
	\item $S$ is a subalgebra and $I$ is a multiplicative ideal.
	\item $\del,\delbar$ vanish on $S$, since both $\del,\delbar$ increases degree but preserves weight. In particular $S$ is a sub-cbba and one can treat $\mathcal{M}$ as a dg-algebra over $S$.
	\item $\delbar\mathcal{M}=\delbar I$, $\del \mathcal{M}=\del I$ and $\del\delbar \mathcal{M}=\del\delbar I$. In particular $\del\delbar I\subset \del I\cap\delbar I$.
\end{itemize}
Now consider the cbba quotient map (by differential ideal)
\[\mathcal{M}\to \mathcal{M}/(I,\del I,\delbar I,\del\delbar I)=\mathcal{M}/(I,\del I,\delbar I)=S/S\cap(\del I,\delbar I).\]
We show this is a quasi-isomorphism and thus $\mathcal{M}$ is quasi-isomorphic to its Bott-Chern cohomology. To see this, the composition
\[
S\to \mathcal{M}\to S/S\cap(\del I,\delbar I)
\]
is surjective and forces the induced map $H(\mathcal{M})\to S/S\cap(\del I,\delbar I)$ to be surjective. On the other hand every class in $H(\mathcal{M})$ has weight equal to degree and thus arises from $S$, so the map must be an isomorphism.

It remains to prove that $\mathcal{M}$ satisfies $\del\delbar$-lemma and that weight $\ge$ degree. To see $\mathcal{M}$ satisfies $\del\delbar$-lemma, it suffices to show that the induced maps
\[
\del,\delbar: H_A\to H_{BC}
\]
are trivial. This follows immediately from the observation that $\del,\delbar$ increases degree but preserves weight. But on both $H_A$ and $H_{BC}$, degree $=$ weight.

	Let $\mathcal{M}_k^{(n)}$ be the lower central series of $\mathcal{M}$. Assume that we have proved weight $\ge$ degree on $\mathcal{M}_{k}^{(n-1)}$. Then it suffices to prove weight $\ge$ degree on new generators. Let us analyze degree $k+1$ generators first. Choose such a generator and represent it by an indecomposable element $x$ in $\mathcal{M}_k^{(n)}$. Then $\del x$ and $\delbar x$ are both decomposables from $\mathcal{M}_{k}^{(n-1)}$. If either of them is non-zero, then $x$ is forced to have a weight (the same as its image) that is $\ge$ degree of $\del x$ (or $\delbar x$) which is bigger than degree of $x$. If $\del x=\delbar x=0$, then either $x$ defines a non-zero class in Bott-Chern cohomology and thus has weight equal to its degree; or it is zero in Bott-Chern, which means that $x=\del\delbar y$ and in particular decomposable by minimality of $\mathcal{M}$, but this contradicts our assumption that $x$ is indecomposable. This proves that degree $k+1$ new generators all have weight $\ge$ degree. Finally for a degree $k$ new generator $x$, either at least one of $\del x,\delbar x$ is non-zero, then it inherits a weight that is bigger than its degree; or both $\del x,\delbar x=0$, then we apply the same argument as before.
\end{proof}
\begin{remark}
	The simply-connectedness condition is technical, the argument works as long as $\mathcal{M}$ carries a canonical nilpotent series fixed by all automorphisms; same below.
\end{remark}
\begin{corollary}\label{formality-fini-coho}
	The same is true if $\mathcal{M}$ is a nilpotent algebra finitely generated in each degree with $H_A^1=0$ and $H_A\mathcal{M}$ is finite dimensional.
\end{corollary}
\begin{proof}
The previous argument goes through, except that now $\aut\mathcal{M}$ is not necessarily an algebraic matrix group. We work around by observing that $\rho$ factors through $h\aut\mathcal{M}$ which is algebraic by \Cref{finite-stage-aut}, therefore we can first lift the grading automorphism to $h\aut\mathcal{M}$. Then notice that $h\aut\mathcal{M}=h\aut\mathcal{M}_k$ for $k$ large. Fix a large $k$, we can lift the grading automorphism to $\aut\mathcal{M}_k$, then to $\aut\mathcal{M}_{k+1}$ since they have the same reductive part. Continue doing so, we get a compatible sequence of liftings to $\aut\mathcal{M}_{k+n}$, which in the end gives a lifting to $\aut\mathcal{M}=\varprojlim_n \aut\mathcal{M}_{k+n}$.
\end{proof}
\begin{remark}
	The above argument in fact proves, under the assumptions of \Cref{formality-fini-coho}, that $\mathcal{M}$ is formal iff $\mathcal{M}_k$ is formal for $k$ large.
\end{remark}

\begin{example}
The minimal algebras in \Cref{example:aut-groups} and \Cref{example:CPn} are formal. The grading automorphism in \Cref{example:CPn} is induced by $\GG_m\times\GG_m\to \GG_m$, $(\alpha,\beta)\mapsto \alpha\cdot\beta$.
\end{example}

\begin{corollary}[cf. \cite{angella2015bott}\cite{milivojevic2024bigraded}]
	All ABC-Massey products vanish on formal algebras.
\end{corollary}
\begin{proof}
	They have wrong weights.
\end{proof}
\begin{corollary}[formality descends]
Let $\KK\subset\LL$ be a characteristic zero field extension, and $\mathcal{M}_\KK$ be a finitely generated nilpotent cbba with $H_A^1=0$ over $\KK$. If $\mathcal{M}_\LL=\mathcal{M}_\KK\otimes_\KK\LL$ is formal over $\LL$, then $\mathcal{M}_\KK$ is formal over $\KK$.
\end{corollary}
\begin{proof}
	The automorphism group $G$ of $\mathcal{M}_\KK$ and $G'$ of $H\mathcal{M}_\KK$ are $\KK$-algebraic groups, equipped with a natural $\KK$-morphism $\rho: G\to G'$. Consider now $x$ a $\KK$-point of $G'$, then $\rho^{-1}(x)$ is a principal homogeneous space (i.e. torsor) of the group $\rho^{-1}(\mathtt{e})$ (all defined over $\KK$); note that $\rho^{-1}(x)$ has an $\LL$-point since $\mathcal{M}_\LL$ is formal over $\LL$. The obstruction to finding a $\KK$-point of $\rho^{-1}(x)$ now lies in
	\[
	H^1(\mathrm{Gal}(\LL/\KK), \rho^{-1}(\mathtt{e})).
	\]
	Since $\rho^{-1}(\mathtt{e})$ is unipotent, this cohomology group vanishes.
\end{proof}
\begin{remark}
	By \Cref{formality-fini-coho} and the remark following it, the same holds if $\mathcal{M}_\KK$ is a nilpotent cbba finitely generated in each degree and has finite dimensional Aeppli cohomology.
\end{remark}
\begin{problem}
	A formal minimal algebra is necessarily a $\del\delbar$-algebra, i.e. it satisfies $\del\delbar$-lemma. Find a criterion for deciding when a $\del\delbar$-algebra is formal (perhaps in terms of automorphism and Hodge structure).
\end{problem}


\appendix
\section{Tensor product}\label{tensor-product}
Here we record results on tensor products of (bounded) zig-zags, whose proof can be found in \cite{stelzig2021structure}. To start with, we organize zig-zags into three families.

\begin{enumerate}[(i)]
	\item (balanced-family) Let $A_n$ $(n\in \ZZ)$ be a zig-zag of balanced type, of length $2n$, as shown below:
	\begin{center}
\begin{tabular}{ c c c c c }
$\begin{tikzcd}[row sep=small, column sep=small]
	\bullet & \bullet \\
	& \bullet & \bullet \\
	&& \bullet
	\arrow[from=1-1, to=1-2]
	\arrow[from=2-2, to=1-2]
	\arrow[from=2-2, to=2-3]
	\arrow[from=3-3, to=2-3]
\end{tikzcd}$ &
$\begin{tikzcd}[row sep=small, column sep=small]
	\bullet & \bullet \\
	& \bullet
	\arrow[from=1-1, to=1-2]
	\arrow[from=2-2, to=1-2]
\end{tikzcd}$ &
 $\bullet$ & 
 $\begin{tikzcd}[row sep=small, column sep=small]
	\bullet \\
	\bullet & \bullet
	\arrow[from=2-1, to=1-1]
	\arrow[from=2-1, to=2-2]
\end{tikzcd}$ &
$\begin{tikzcd}[row sep=small, column sep=small]
	\bullet \\
	\bullet & \bullet \\
	& \bullet & \bullet
	\arrow[from=2-1, to=1-1]
	\arrow[from=2-1, to=2-2]
	\arrow[from=3-2, to=2-2]
	\arrow[from=3-2, to=3-3]
\end{tikzcd}$\\

  $A_{-2}$ & $A_{-1}$ &$A_0$ & $A_1$ & $A_2$
\end{tabular}
\end{center}

 \item (vertical-family) Let $B_n$ $(n\in \NN_+)$ be a zig-zag of vertical type of length $(2n-1)$, as shown below:
 \begin{center}
 \begin{tabular}{ c c c }
 $\begin{tikzcd}[row sep=small, column sep=small]
	\bullet \\
	\bullet
	\arrow[from=2-1, to=1-1]
\end{tikzcd}$ &
$\begin{tikzcd}[row sep=small, column sep=small]
	\bullet \\
	\bullet & \bullet \\
	& \bullet
	\arrow[from=2-1, to=1-1]
	\arrow[from=2-1, to=2-2]
	\arrow[from=3-2, to=2-2]
\end{tikzcd}$ &
$\begin{tikzcd}[row sep=small, column sep=small]
	\bullet \\
	\bullet & \bullet \\
	& \bullet & \bullet \\
	&& \bullet
	\arrow[from=2-1, to=1-1]
	\arrow[from=2-1, to=2-2]
	\arrow[from=3-2, to=2-2]
	\arrow[from=3-2, to=3-3]
	\arrow[from=4-3, to=3-3]
\end{tikzcd}$\\
$B_1$ & $B_2$ & $B_3$
 \end{tabular}	
 \end{center}

 \item (horizontal-family) Let $C_n$ $(n\in \NN_+)$ be a zig-zag of horizontal type of length $(2n-1)$, as shown below:
 \begin{center}
 \begin{tabular}{ c c c }
 $\begin{tikzcd}[row sep=small, column sep=small]
	\bullet & \bullet
	\arrow[from=1-1, to=1-2]
\end{tikzcd}$ &
$\begin{tikzcd}[row sep=small, column sep=small]
	\bullet & \bullet \\
	& \bullet & \bullet
	\arrow[from=1-1, to=1-2]
	\arrow[from=2-2, to=1-2]
	\arrow[from=2-2, to=2-3]
\end{tikzcd}$ &
$\begin{tikzcd}[row sep=small, column sep=small]
	\bullet & \bullet \\
	& \bullet & \bullet \\
	&& \bullet & \bullet
	\arrow[from=1-1, to=1-2]
	\arrow[from=2-2, to=1-2]
	\arrow[from=2-2, to=2-3]
	\arrow[from=3-3, to=2-3]
	\arrow[from=3-3, to=3-4]
\end{tikzcd}$\\
$C_1$ & $C_2$ & $C_3$
 \end{tabular}	
 \end{center}
\end{enumerate}
These families are distinguished by Dolbeault cohomology groups, which pick out the ends of the zig-zags. We have:
\begin{enumerate}[(i)]
	\item $H_\del(A_n)=\KK$ and $H_\delbar(A_n)=\KK$,
	\item $H_\del(B_n)=\KK^2$ and $H_\delbar(B_n)=0$,
	\item $H_\del(C_n)=0$ and $H_\delbar(C_n)=\KK^2$.
\end{enumerate}

\begin{theorem}[cf. \cite{stelzig2021structure}]
	There are equivalences, module squares and degree shifts:
	\begin{enumerate}[(i)]
		\item $A_i\otimes A_j\equiv A_{i+j}$ for all $i,j$;
		\item $A_i\otimes B_j\equiv B_j$ and $A_i\otimes C_j=C_j$ for all $i,j$;
		\item $B_i\otimes B_j\equiv2 B_i$, $C_i\otimes C_j=2C_j$ for $i\le j$;
		\item $B_i\otimes C_j\equiv 0$ for all $i,j$.
	\end{enumerate}
	The appropriate degrees and locations of these zig-zags can be easily read off from applying K\"unneth theorem to Dolbeault cohomologies.
\end{theorem}
\begin{proof}[Sketch of proof]
	(iv) can be proved using the vanishing of the Dolbeault groups. (ii) and (iii) can be deduced from (i) by noticing B,C-types are quotients of A-types with a one-dimensional kernel. Finally (i) can be obtained by iteratively applying $A_1\otimes -$ and $A_{-1}\otimes -$; by symmetry it suffices to understand the behavior of $A_1\otimes-$, which can be computed explicitly.
\end{proof}


\bibliographystyle{alpha}
\bibliography{ref}

\begin{thebibliography}{DGMS75}

\bibitem[AT15]{angella2015bott}
Daniele Angella and Adriano Tomassini.
\newblock {On Bott--Chern cohomology and formality}.
\newblock {\em Journal of Geometry and Physics}, 93:52--61, 2015.

\bibitem[BL05]{block-lazarev}
J.~Block and A.~Lazarev.
\newblock {Andr\'e-Quillen cohomology and rational homotopy of function
  spaces}.
\newblock {\em Advances in Mathematics}, 193(1):18--39, 2005.

\bibitem[DGMS75]{dgsm}
Pierre Deligne, Phillip Griffiths, John Morgan, and Dennis Sullivan.
\newblock {Real homotopy theory of K{\"a}hler manifolds}.
\newblock {\em Inventiones mathematicae}, 29(3):245--274, 1975.

\bibitem[KQ20]{khovanov2020faithful}
Mikhail Khovanov and You Qi.
\newblock {A faithful braid group action on the stable category of
  tricomplexes}.
\newblock {\em SIGMA. Symmetry, Integrability and Geometry: Methods and
  Applications}, 16:019, 2020.

\bibitem[MS24]{milivojevic2024bigraded}
Aleksandar Milivojevi{\'c} and Jonas Stelzig.
\newblock {Bigraded notions of formality and Aeppli-Bott-Chern-Massey
  products}.
\newblock {\em Communications in Analysis and Geometry}, 32(10):2901--2933,
  2024.

\bibitem[Sim92]{simpson1992higgs}
Carlos~T Simpson.
\newblock {Higgs bundles and local systems}.
\newblock {\em Publications Math{\'e}matiques de l'IH{\'E}S}, 75:5--95, 1992.

\bibitem[Ste21]{stelzig2021structure}
Jonas Stelzig.
\newblock On the structure of double complexes.
\newblock {\em Journal of the London Mathematical Society}, 104(2):956--988,
  2021.

\bibitem[Ste25]{stelzig2025pluripotential}
Jonas Stelzig.
\newblock Pluripotential homotopy theory.
\newblock {\em Advances in Mathematics}, 460:110038, 2025.

\bibitem[Sul77]{infinitesimal}
Dennis Sullivan.
\newblock Infinitesimal computations in topology.
\newblock {\em Publications Math\'ematiques de l'IH\'ES}, 47:269--331, 1977.

\end{thebibliography}


\end{document}